\documentclass[final]{siamltex}
\usepackage {epsfig,amsmath,euscript}
\usepackage[latin1]{inputenc}
\usepackage[english]{babel}
\usepackage{color}
\usepackage{amsmath}
\usepackage{amssymb}
\usepackage{ae}
\usepackage{subfigure}
\usepackage[T1]{fontenc}
\usepackage{graphicx}
\usepackage{epstopdf}
\usepackage{color}
\definecolor{rred}{rgb}{0.7,0.0,0.2}
\definecolor{bblue}{rgb}{0.2,0.0,0.7}

\newcommand{\be}{\begin{eqnarray}}
\newcommand{\ee}{\end{eqnarray}}
\newcommand{\ba}{\begin{align}}
\newcommand{\ea}{\end{align}}
\newcommand{\bi}{\begin{itemize}}
\newcommand{\ei}{\end{itemize}}

\newcommand{\R}{\mathbb R}

\newcommand{\beq}[1]{\begin{equation} \label{#1}}
\newcommand{\eeq}{\end{equation}}
\newcommand{\beqa}{\begin{eqnarray}}
\newcommand{\eeqa}{\end{eqnarray}}
\newcommand{\bal}{\begin{align}}
\newcommand{\eal}{\end{align}}
\newcommand{\bsub}{\begin{subequations}}
\newcommand{\esub}{\end{subequations}}
\newcommand{\eqlab}[1]{\label{eq:#1}}
\renewcommand{\eqref}[1]{(\ref{eq:#1})}
\newcommand{\eqsref}[2]{(\ref{eq:#1}) and~(\ref{eq:#2})}

\newcommand{\figref}[1]{Fig.~\ref{fig:#1}}
\newcommand{\figlab}[1]{\label{fig:#1}}

\newcommand{\secref}[1]{section~\ref{sec:#1}}
\newcommand{\seclab}[1]{\label{sec:#1}}

\newcommand{\exmref}[1]{Example~\ref{example:#1}}
\newcommand{\exmlab}[1]{\label{example:#1}}
\newcommand{\remref}[1]{Remark~\ref{remark:#1}}
\newcommand{\remlab}[1]{\label{remark:#1}}
\newcommand{\corref}[1]{Corollary~\ref{Corollary:#1}}
\newcommand{\corlab}[1]{\label{Corollary:#1}}

\newcommand{\thmref}[1]{Theorem~\ref{theorem:#1}}
\newcommand{\thmlab}[1]{\label{theorem:#1}}

\newcommand{\lemmalab}[1]{\label{lemma:#1}}

\title{An iterative method for the approximation of fibers in slow-fast systems}

\author{K. Uldall Kristiansen, M. Br{\O}ns and J. Starke\thanks{Department of Applied Mathematics and Computer Science, Technical University of Denmark, 2800 Kgs. Lyngby, DK. K. Uldall Kristiansen was funded by a H. C. {\O}rsted post doc grant.}}
\begin{document}

\maketitle


 


\begin{abstract}
In this paper we extend a method for iteratively improving slow manifolds
so that it also can be used to approximate the fiber directions. 
The extended method is applied to general finite dimensional real analytic
systems where we obtain exponential estimates of the tangent spaces to the
fibers. The method is demonstrated
on the Michaelis-Menten-Henri model and the Lindemann mechanism. The latter example also serves to demonstrate the method on a slow-fast system in non-standard slow-fast form. Finally, we extend the method further
so that it also approximates the curvature of the fibers.
\end{abstract} 


\pagestyle{myheadings}
\thispagestyle{plain}
%
 \section{Introduction}
Singularly perturbed systems involving different scales in time and/or space arise in a wide variety of scientific problems. Important examples include: meteorology and short-term weather forecasting \cite{lor2,lor1,tem1}, molecular physics and the Born-Oppenheimer approximation \cite{McQ1}, chemical enzyme kinetics and the Michaelis-Menten mechanism \cite{micmen1}, predator-prey and reaction-diffusion models \cite{Mur1}, the evolution and stability of the solar system \cite{las1,las2} and the modeling of tethered satellites \cite{kri1,kri2}. These systems can also be ``artificially constructed'' by a partial scaling of variables near a bifurcation \cite{romtur1}. The main advantage of identifying slow and fast variables is dimension reduction by which all the fast variables are \textit{slaved} to the slow ones through the \textit{slow manifold}. Dimension reduction is one of the main aims and tools for a dynamicist and the elimination of fast variables is very useful in for example numerical computations. 
Since 
fast variables require more computational effort and evaluations, this reduction often bridges the gap between tractable and intractable computations. An example of this is the long time ($Gyear$s) integration of the solar system, see \cite{las1,las2}. See also \cite{cotrei1} for a numerical treatment of slow-fast systems.

In this paper, we consider standard slow-fast systems of the form
\begin{eqnarray}
 \dot x &=&\epsilon X(x,y),\eqlab{xys0}\\
 \dot y &=& Y(x,y),\nonumber\\
 \dot{()} &=&\frac{d}{dt},\nonumber
\end{eqnarray}
with a small parameter $\epsilon$. The vector-fields $X$ and $Y$ will be analytic in $x$ and $y$ but may in
general also depend upon $\epsilon$. For simplicity we shall, however, always suppress the $\epsilon$-dependency only making reference to it when needed.


Let $M_0=\{(x,y)\vert Y(x,y)=0\}$. We will return to this set $M_0=M_0(\epsilon)$, which depends upon $\epsilon$, later, but we will first consider the constrained set $M_0(0)=M_0\vert_{\epsilon=0}$ instead. This is the critical manifold \cite{jon1} and it is a fixed point set of \eqref{xys0}$\vert_{\epsilon=0}$ and therefore invariant. If these fixed points are hyperbolic: 
{\begin{eqnarray}
\lambda \in \sigma\left(\partial_y Y(x,y)\vert_{M_0(0)}\right) \Rightarrow \vert \text{Re}\,\lambda\vert \ge \lambda_0 \ne 0,\eqlab{M0hyp}
\end{eqnarray}}
with $\lambda_0$ independent of $\epsilon$, then $M_0(0)$ is said to be normally hyperbolic. Here $\sigma(A)$ denotes the
spectrum of an operator $A$. Moreover, $\partial_z$ is used to denote the partial derivatives $\frac{\partial
}{\partial z}$, and we will continue to use this symbol regardless of what
object is being differentiated. 
In the case of normally hyperbolicity, and when $X$ and $Y$ are also smooth in $\epsilon$, then Fenichel's theory \cite{fen1}, \cite[Theorem 2, p.8]{jon1} applies and one can conclude that there exists an $\epsilon_0$ such that $M_0(0)$ perturbs to an invariant normally hyperbolic set $M=M(\epsilon)$ for all $\epsilon\le \epsilon_0$. Moreover, to each point $x\in M$ there exists stable/unstable fibers where points contract exponentially fast as $t\rightarrow \infty$/$t\rightarrow -\infty$ towards the forward/backward flow of $x$. The unions of these fibers make up the local stable and unstable manifolds of $M$ which are diffeomorphic to the unperturbed ones. 
{Fenichel's slow manifolds are examples of normally hyperbolic invariant manifolds. These are global objects. Slow manifolds are also examples of center manifolds \cite{murd1,har1,tak1}. However, center manifolds may also be examples of non-slow, local, normally hyperbolic invariant manifolds. }

If $M_0=M_0(\epsilon)$ is not normally hyperbolic at $\epsilon=0$ but only satisfies the weaker condition
 \begin{eqnarray}
 \lambda \in \sigma\left(\partial_y Y(x,y)\vert_{M_0(\epsilon)}\right) \Rightarrow \vert \lambda\vert \ge \lambda_0>0,\eqlab{M0fast}
 \end{eqnarray}
 then in general there exists no invariant manifold near $M_0$ due to resonances \cite{man1,geller1,geller2,lor3}. Note that the condition \eqref{M0fast} implies that $M_0=M_0(\epsilon)$ can be written as a graph $y=\eta_0(x)$. The condition \eqref{M0fast} is the meaning of $y$ being fast \cite{mac1}. In analytic systems, however, the destruction by resonances only manifests itself in exponentially small error terms \cite[Lemma 1]{nei87}, \cite{kri3}. Indeed the following statement holds true: {There exists an $\epsilon_0$ such that for all $\epsilon\le \epsilon_0$ there exists a graph $M=\{y=\eta(x)\}$ which is exponentially close $\mathcal O(e^{-c/\epsilon})$ to being invariant.  Here $c$ is independent of $\epsilon$. Moreover, $M$ is $\epsilon$-close to $M_0=\{y=\eta_0(x)\}$. } If the slow-fast system is Hamiltonian then $M$ can be made symplectic on which a (formally) reduced Hamiltonian system can be defined. These statements hold true even when $X$ and $Y$ depend non-smoothly on $\epsilon$. {We will 
return to this in \secref{bg} 
where we also consider an example (\exmref{nex}) where $Y$ depends non-smoothly on $\epsilon$}. Even stronger results hold true in the case \eqref{M0fast} when considering normally elliptic $M_0$ with $\sigma(\partial_y Y(x,y)\vert_{M_0(\epsilon)})\subset i\R$ and one fast degree of freedom \cite{geller1}. One can then use averaging to obtain a whole foliation of 
exponentially accurate invariant manifolds. These are, however, not all slow and the averaging principle does not extend to several fast variables due to resonances between these. 

{The reference \cite{ioo1} considers a related scenario of an analytic vector-field near an equilibrium. The linearized system is assumed to be split in two invariant subspaces $E_0$ and $E_1$. Under certain diophantine conditions on the eigenvalues in $E_0$ the reference shows that there is a graph slaving the variables in $E_1$ to those in $E_0$ which is exponentially close to invariance. This result is local in the variables in $E_0$ \cite[Eq. (8), Theorem 1]{ioo1}. Besides the diophantine condition a crucial condition is, as for the references above, the requirement about analyticity. This condition cannot be relaxed. }

The results on slow manifolds motivate the development of reduction methods for the approximation of these invariant or almost invariant objects. \textit{The method of straightening out} (SO henceforth) used in \cite{kri3} is an example of a reduction method that successively provides better approximations to invariant manifolds. The method was suggested by MacKay in \cite{mac1} but it is identical to the method suggested by Fraser and Roussel in \cite{fra1,fra2}. In \cite{kap3} this method is also
 referred to as the iterative method of Fraser and Roussel. In \cite{gou1} it is called the invariance equation method. The use of different names is unfortunate but from our view-point, which is due to MacKay, we find SO more descriptive. MacKay's description also highlights properties that are usually not attributed to the method. The power of the SO method is four-fold. (i): It leads to exponential accurate slow manifolds. {(ii): It can written in a form  (see \eqref{fraser} below) that only involves the vector-field, hence avoiding the lengthy details of asymptotic expansions}. (iii): It does not require smoothness of $X$ and $Y$ in $\epsilon$. (iv): The slow manifold includes nearby equilibria. There are, however, several alternatives to the SO method. We name a few others: The intrinsic low-dimensional manifold (ILDM) method of Maas and Pope \cite{maa1}, the zero-derivative principle (ZDP) \cite{kap1,kap2}, and the computational singular perturbation (CSP) method initially due to Lam and Goussis \cite{lam1,lam2}, and later thoroughly analyzed by Zagaris and co-
workers \cite{kap4}. The ILDM method is based on the Jacobian of the vector-field and 
partitions this at each point into a fast and a slow component based on spectral gaps of the Jacobian. The ILDM approximation to the slow manifold is then defined as the locus of points where 
the vector-field lies entirely in the slow subspace. In general, this only gives an approximation that agrees up to $\mathcal O(\epsilon)$ \cite{kap3}. Nevertheless, the method is still quite powerful as it can be used in systems where a small parameter may not be directly available. In the ZDP method an $\mathcal O(\epsilon^n)$-accurate
approximation to the slow manifold is obtained as the locus of points where the $(n+1)$th time derivative of the fast variables vanishes. This method has been used in an equation-free setting in \cite{kap1}. The CSP method also provides $\mathcal O(\epsilon^n)$-approximations of the slow manifolds \cite{kap4} and it is, as the ILDM, based on the decomposition of the tangent space into fast and slow subspaces. 

The use of the CSP method is not restricted to problems where slow and fast variables have been properly identified as in \eqref{xys0}. Part of the outcome of the CSP method is the identification of fast and slow subspaces. This particularly means that when applying the CSP method to system \eqref{xys0}, it leads to an approximation of the tangent spaces  to the fibers through a set of basis vectors, see e.g. \cite{kap4}. In \cite{gou1} a ``CSP-like'' method is nevertheless developed as an extension of the SO method, it also being capable of identifying the fast and slow subspaces. It is shown \cite[App. A]{gou1} that this method leads to a more efficient algorithm when compared to the CSP method. On the other hand, this method does not enjoy the same properties as the usual SO method since, as the CSP method, it also requires higher order partial derivatives of the vector-field for improvements beyond leading order.  We will in this paper show that it is also possible to
approximate the tangent spaces of the fibers by adding an extra step to the SO method without introducing the need for higher order partial derivatives of the vector-field.



\subsection{Aims of the paper} The main aim of this paper is to present a simple procedure for the approximation of the tangent spaces of the fibers. We will extend the interpretation of this approximation so that it also has meaning for non-hyperbolic slow manifolds where Fenichel's theory does not apply. We follow similar lines as those developed in \cite{rob89,ajr1} approximating related spaces in systems near equilibria. We will refer to this method as the \textit{SOF method} - the extra \textit{F} has been added to SO to indicate that the approximation of the fiber directions is build in as an extension of the original SO method. 
The extension will enjoy similar properties to the traditional SO method. (i): It leads to exponential estimates. (ii): It only involves the vector-field and its Jacobian, in contrast to e.g. the CSP method. (iii): It does not require smoothness of $X$ and $Y$ in $\epsilon$. (iv): The spaces are exact at equilibria (see also remark \remref{mu1prop} below for further clarification on this part). Moreover, we will extend our technique to approximate curvatures.



\subsection{Applications} 
As opposed to \cite{kri3} the applications we have in mind are primarily
for normally hyperbolic slow manifolds, where the fibers provide the
directions of the stable and unstable manifolds along which the solutions
relax to respectively escape the slow manifold. However, our results in
\thmref{thm1} still hold true for e.g. the normally elliptic case by providing
coordinates in which the slow dynamics become \textit{almost} independent
of the fast variables to linear order (see also \eqref{fintp} for further clarification). {We highlight that a related scenario is considered in \cite{rob89} which considers dynamics near equilibria but does not restrict to normally hyperbolic center manifolds. Indeed, the results of \cite{rob89} apply to invariant manifolds arising from other means such as those from Lyapunov center theorem in Hamiltonian systems and the almost invariant ones described by \cite{ioo1}. } 
%

{We have in \cite{kri4} begun an analysis of the numerical implementation of the SOF method for the computation of orbits connecting to and departing from canard segments on saddle-type slow manifolds. Here direct integration is futile.  Such segments are covered by the Exchange Lemma \cite{jon1} and appear in many applications, such as the Fitz-Hugh-Nagumo model \cite{guck6} and the van-der Pol equations \cite{guck5}. The idea is to use the SOF method to obtain by truncation a splitting of the problem, allowing us to first compute the canard segment itself, and then follow this by computing the fast part initially connecting to it and finally departing from it, using collocation on the fast space only on short $\mathcal O(1)$-time scales. A nice property of this method is that it does not increase in complexity as $\epsilon$ decreases. By considering a model for reciprocal inhibition with two slow and two fast variables $n_s=2=n_f$, we have compared our results with the results 
from using the collocation method suggested in \cite{guc3}. This looks promising and we aim to submit \cite{kri4} in the near future. }
 
 \subsection{Structure of paper} 
 After introducing our notation we will in \secref{bg} provide more background on the topic and include short descriptions of the traditional SO method and its new extension. Then in \secref{main} we present our main results on the approximation of the tangent spaces of the fibers. The main theorem is proven in \secref{prvthm1}. We apply the results to the Michaelis-Menten-Henri model in \secref{appl} before we in \secref{curve} extend our method so that it also approximates the curvature of the fibers. In principle higher order
effects can also be accounted for, but this introduces a certain degree of
complexity. In this paper we will therefore focus most of our effort on
demonstrating the first part of the method which seeks to estimate the
tangent spaces. Once this approach has been
established and demonstrated on the Michaelis-Menten-Henri model, we will
consider removing the part of the slow vector-field which is quadratic in
the fast variable, hence approximating the curvature of the fibers. One of
the reasons for choosing the Michaelis-Menten-Henri model as our example is
that all the calculations can be done explicitly. But moreover, it also
allows for {comparison} with the results in \cite{kap4} from the application
of the CSP method. {While our main focus will be on the standard slow-fast form \eqref{xys0}, we will nevertheless before our conclusion section, demonstrate on the Lindemann mechanism \cite{gou1} how the SOF method applies when the slow-fast system is not presented in slow-fast form \eqref{xys0}.} We also here compare our results with computations based on the CSP method.
 
 \subsection{Notation and preliminaries} 
We believe our result are best presented using sequences of transformations. We believe this makes the proofs simpler. In particular, the need for diminishing the domains becomes clearly apparent. Moreover, we believe that the method then fits nicely within normal form theory. See e.g. \cite{guc1} section 3.3, where one (using averaging) also seeks to remove a current error, the result of which is to introduce a new \textit{but smaller error}. On the downside, however, we have to deal with several different variables, being the consequences of the applications of these different successively defined transformations. 
For this we will use subscripts and superscripts on the variables. Below in \secref{bg}, where we present the SO method, we will, for example, successively introduce transformations of the form $y_{i+1}\mapsto y_i$. The slow variables will not be transformed during this step. {The purpose of each transformation $y_{i+1}\mapsto y_i$ is to push the resulting level set $y_{i+1}=0$ closer to invariance.} This sequence will be stopped at $y_{N_1}$. To avoid unnecessary clutter in the following we proceed by dropping the subscript $N_1$ and introduce $y\mapsto y_0$ as the composition of all the transformations $y_{i+1}\mapsto y_i$, $i=0,\ldots,N_1-1$. When we follow this by the new extension in \secref{sofsec} and \secref{prvthm1}, and thus start from $(x_0,y_{N_1})=(x_0,y)$, we introduce transformations $x_{i+1}\mapsto x_i$, $i=0,\ldots,N_2-1$, of the slow variables. The purpose of these transformations is to eliminate the linear dependency on $y$ in the slow part of the vector-field. The product of these 
finitely 
many transformations again generates a transformation $x_{N_2}\mapsto x_0$. Finally, in \secref{curve} we proceed by applying 
another sequence of transformations of 
the 
slow variables 
starting from $x_{N_2}$ that seek to remove terms quadratic in $y$ in the slow part of the vector-field. For simplicity, we will drop the subscript $N_2$ and replace it by a superscript $0$ so that $x^0=x_{N_2}$ is the starting point for our final iteration $x^{i+1}\mapsto x^i$, $i=0,\ldots,N_3-1$ (not to be confused with a power). 

Superscripts will also be used on the computed functions $\eta$,$\phi$ and $\psi$, the former describing the slow manifold, the latter two describing the fibers, to denote partial sums:
\begin{eqnarray}
 \eta^n = \sum_{i=1}^n \eta_i,\quad n\ge 1,\eqlab{etapn}
\end{eqnarray}
and
\begin{eqnarray*}
 \quad \phi^n = \sum_{i=0}^n \phi_i,\quad \psi^n = \sum_{i=0}^n\psi_i,\quad n\ge 0.
\end{eqnarray*}
We do not believe this will cause any unnecessary confusion as this notation will only be used on these three functions. 
The superscripts refer to the order of accuracy as $\eta^n$ e.g. will introduce a remainder of $\mathcal O(\epsilon^{n+1})$. The subscripts on the functions are similarly related to their order with respect to $\epsilon$:
\begin{eqnarray*}
 \eta_i = \mathcal O(\epsilon^i),\quad \phi_i = \mathcal O(\epsilon^i),\quad \psi_i =\mathcal O(\epsilon^i).
\end{eqnarray*}
Note that the sum in $\eta^n$ \eqref{etapn} starts from $i=1$ since our starting point (\eqref{xys1} below) will based on the deviations $y_0$ from $y=\eta_0(x)$. (See \remref{sofn} to appreciate the convenience of this choice). We finally point out that $\partial_{x_i}$ and $\partial_{x^i}$ will be denoted by $\partial_x$ and that we will use the
alternative notation $(x)_i$, $1\le i\le n$, for the $i$th component of a
vector $x\in \mathbb R^{n}$. 

Let $(\mathcal X,\Vert \cdot \Vert_{\mathcal X})$ and $(\mathcal Y,\Vert
\cdot \Vert_{\mathcal Y})$ be real Banach spaces, and $\mathcal X_{\mathbb
C}=\mathcal X\oplus i\mathcal X$ respectively $\mathcal Y_{\mathbb
C}=\mathcal Y\oplus i\mathcal Y$ their complexifications with norms $\Vert
x_1+ix_2\Vert_{\mathcal X_{\mathbb C}}=\Vert x_1\Vert_{\mathcal X}+\Vert
x_2\Vert_{\mathcal X}$ and $\Vert y_1+iy_2\Vert_{\mathcal Y_{\mathbb
C}}=\Vert y_1\Vert_{\mathcal Y}+\Vert y_2\Vert_{\mathcal Y}$. Here we are
primarily thinking of $\mathcal X$ and $\mathcal Y$ as Euclidean spaces.

We will from now on denote all norms,
including operator norms, by $\Vert \cdot \Vert$. 
Then $f:\mathcal U_{\mathbb C}\rightarrow
\mathcal Y_{\mathbb C}$, with $\mathcal U_{\mathbb C}$ an open subset of
$\mathcal X_{\mathbb C}$, is analytic if it is continuously differentiable.
That is if there exists a continuous derivative $\partial_x f:\mathcal
U_{\mathbb C}\rightarrow L(\mathcal X_{\mathbb C},\mathcal Y_{\mathbb C})$,
the Banach space of complex linear operators from $\mathcal X_{\mathbb C}$
to $\mathcal Y_{\mathbb C}$ equipped with the operator norm, satisfying the
following condition
\[\Vert f(x+h)-f(x)-\partial_x f(x)(h)\Vert = \mathcal O(\Vert h\Vert^2).\]
By real analytic we will mean analytic and real when the arguments are real. The higher order derivatives can be defined inductively and $\partial_x^n f$ becomes a map
\begin{eqnarray*}
 \partial_x^n f:\mathcal U_{\mathbb C}\rightarrow L^n(\mathcal X_{\mathbb C},\mathcal Y_{\mathbb C}),
\end{eqnarray*}
from $\mathcal U_{\mathbb C}$ into the Banach space $L^n(\mathcal X_{\mathbb C},\mathcal Y_{\mathbb C})$ of all bounded, $n$-linear maps from $\mathcal X_{\mathbb C}\times \cdots \times \mathcal X_{\mathbb C}$ ($n$ times) into $\mathcal Y_{\mathbb C}$. See \cite[App. A]{pos1} for a reference on analytic function theory in Banach spaces. 

When $\mathcal U$ is an open subset of $\mathcal X$ then we define $\mathcal U+i\chi$ to be the open complex $\chi$-neighborhood of $\mathcal U$:
\begin{eqnarray*}
 \mathcal U+i\chi = \{x\in \mathcal X_{\mathbb C}\vert d_{\mathcal X_{\mathbb C}}(x,\mathcal U) < \chi\},
\end{eqnarray*}
where $d_{\mathcal X_{\mathbb C}}$ is the metric induced from the Banach norm $\Vert\cdot \Vert$.

 We frequently need the following Cauchy estimate:
\begin{lemma}\lemmalab{cest} \cite{pos1}
Assume that $f:\mathcal U_{\mathbb C}\rightarrow \mathcal Y_{\mathbb C}$ is analytic and that $f$ is bounded on the $\mathcal X_{\mathbb C}$-open ball $B_\xi(x_0)\subset \mathcal U_{\mathbb C}$ centered at $x_0\in \mathcal U_{\mathbb C}$ and with radius $\xi< d_{\mathcal X_{\mathbb C}}(x_0,\partial \mathcal U_{\mathbb C})$. Then
\begin{eqnarray}
 \Vert \partial_x f(x_0)\Vert \le \frac{\sup_{x\in B_\xi(x_0) }\Vert f(x)\Vert}{\xi}.\eqlab{cest}
\end{eqnarray}
\end{lemma}
\begin{remark}\remlab{cauchyx0}
Consider $f:\mathcal U+i\chi\rightarrow \mathcal Y_{\mathbb C}$ analytic and bounded. Then we can apply this estimate to any $x_0\in \mathcal U+i(\chi-\xi)$ to obtain:
\begin{eqnarray*}
\sup_{x_0\in \mathcal U+i(\chi-\xi)}\Vert \partial_x f(x_0)\Vert \le\frac{\sup_{x\in \mathcal U+i\chi}\Vert f(x)\Vert}{\xi},
\end{eqnarray*}
which we will write compactly as
\begin{eqnarray*}
\Vert \partial_x f\Vert_{\chi-\xi} \le\frac{\Vert f(x)\Vert_{\chi}}{\xi}.
\end{eqnarray*}
This is the form of Cauchy's estimate that we will be using. Similarly, we will by $\Vert \cdot \Vert_{\chi,\nu}$ denote the sup-norm taking over the domain $(\mathcal U+i\chi) \times (\mathcal V+i\nu)$ of $(x,y)$.

Note also that the norm on the left hand side of \eqref{cest} is the operator norm on $L(\mathcal X_{\mathbb C},\mathcal Y_{\mathbb C})$ of complex bounded linear operators, while the norm on the right hand side is the norm on $\mathcal Y_{\mathbb C}$.
\end{remark}

\begin{remark}
  We write a $m$-linear form such as $\partial_x^m f(x)\in
L^m(\mathcal X_{\mathbb C},\mathcal Y_{\mathbb C})$ evaluated diagonally
on $h\in \mathcal X$ as $\partial_x^m f(x)h^m$. With this notation 
Taylor's formula reads:
\begin{eqnarray*}
 f(x+h) &=& f(x)+\partial_x f(x)h+\cdots +\frac{1}{(n-1)!}\partial_x^{n-1} f(x)h^{n-1}\\
&\quad &+\int_0^1 \frac{(1-s)^{n-1}}{(n-1)!}\partial_x^n f(x+sh)h^n ds,\quad \text{whenever}\quad x+sh \in \mathcal U_{\mathbb C} \quad \text{for all}\quad s\in [0,1],\eqlab{Taylor}
\end{eqnarray*}
or more compactly as
\begin{eqnarray*}
 f(x+h) &=& f(x)+\partial_x f(x)h+\cdots +\frac{1}{(n-1)!}\partial_x^{n-1} f(x)h^{n-1}+\mathcal O(h^n),
\end{eqnarray*}
here introducing the big-O notation for the integral remainder, which is bounded by $$\frac{\Vert h \Vert^n}{n!}\sup_{0\le
s\le 1} \Vert \partial_x^n f(x+sh)\Vert.$$ 
\end{remark}

\section{Further background}\seclab{bg}

\subsection{The SO method}\seclab{somethod}


We consider \eqref{xys0} and assume that condition \eqref{M0fast} holds true. By the analytic implicit function theorem the set $M_0=\{(x,y)\vert Y(x,y)=0\}$ can therefore be represented as a graph $M_0=\{(x,y)\vert y=\eta_0(x)\}$ with $\eta_0=\eta_0(x)$ analytic. We then introduce $(x,y)=(x_0,y_0+\eta_0(x_0))$ to transform these equations into
\begin{eqnarray}
 \dot x_0 &=&\epsilon X_0(x_0,y_0)\equiv \epsilon X(x_0,\eta_0+y_0),\eqlab{xys1}\\
 \dot y_0 &=&Y_0(x_0,y_0) = \rho_0(x_0)+A_0(x_0)y_0+R_0(x_0,y_0) \equiv -\epsilon \partial_x \eta_0 X(x_0,\eta_0(x_0)+y_0)+Y(x_0,\eta_0(x_0)+y_0),\nonumber
\end{eqnarray}
$(\epsilon X_0,Y_0)$ denoting the new vector-field, and by Taylor expanding $Y_0$ about $y_0=0$ we identify the following relevant functions
\begin{eqnarray*}
\rho_0 &=&Y_0(x_0,0)= -\epsilon \partial_{x} \eta_0 X(x_0,\eta_0),\\
A_0&=&\partial_y Y_0(x_0,0)=\partial_{y} Y(x_0,\eta_0)-\epsilon \partial_{x} \eta_0 \partial_{y} X(x_0,\eta_0).
\end{eqnarray*}
The function $R_0=\mathcal O(y_0^2)$ is the remainder from the Taylor expansion. It can be described in the following way using the integral remainder formula
\begin{eqnarray*}
R_0(x_0,y_0) &=& \int_0^1 (1-s) \left(-\epsilon \partial_x  \eta_0 \partial_y^2  X(x_0,  \eta_0(x_0) +sy_0)+\partial_y^2   Y(x_0,  \eta_0(x_0)+sy_0)\right)y_0^2 ds = \mathcal O(y_0^2). 
\end{eqnarray*}
{Here $y_0$ describes displacements from the graph $M_0=\{(x,y)\vert y=\eta_0(x)\}$ which has now been transformed to $\{y_0=0\}$ and which we continue to denote by $M_0$.} See \figref{y0defeta1} (a). 
The manifold $M_0$ is not invariant since $Y_0\vert_{y_0=0} =\rho_0$, but it is close to being invariant as $\rho_0=\mathcal O(\epsilon)$ is small. Moreover, $y_0$ is \textit{fast} transverse to $y_0=0$ as {$\Vert A_0^{-1}\Vert \gg \epsilon$}, cf. \eqref{M0fast}, and $Y_0$ therefore ``varies'' $\mathcal O(1)$ with respect to displacements in $y_0$ from $y_0=0$.
\begin{figure}[h!]
\begin{center}
\subfigure[]{\includegraphics[width=.86\textwidth]{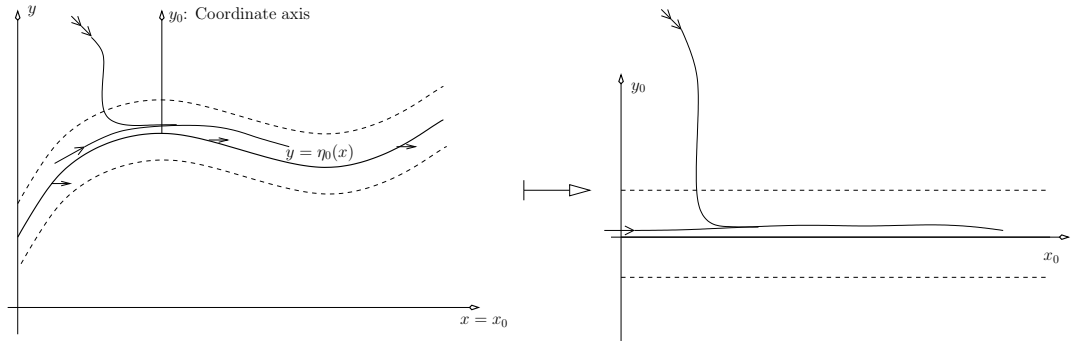}}
\subfigure[]{\includegraphics[width=.5\textwidth]{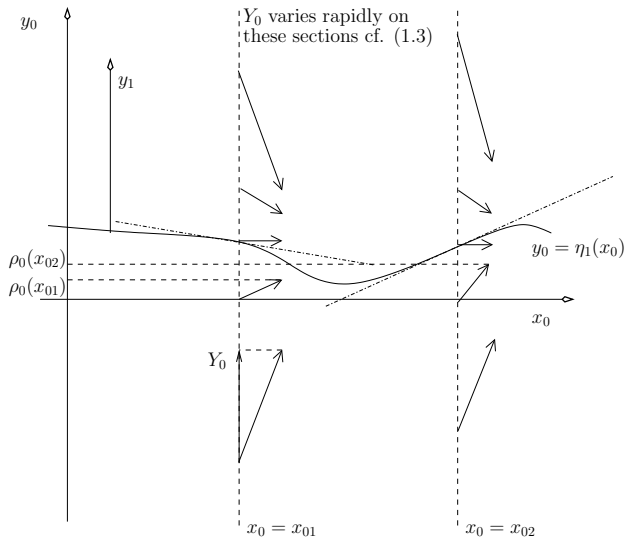}}
\end{center}
\caption{(a): \textit{Straightening out} $\eta_0$. The geometrical interpretation of $y_0$ as the deviation from $y=\eta_0(x)$. Note how, by definition, the vector-field restricted to $y=\eta_0(x)$ (indicated by small vectors) only points in the $x$-direction. (b): The geometrical interpretation of the SO method. On each vertical section $x_0=\text{fixed}$ we can solve for $Y_0=0$ with respect to $y_0$. This gives $\eta_1=\eta_1(x)$. We have also provided the interpretation of $\rho_0$ on the ordinate. The new error $\rho_1$, obtained after having straightened out $\eta_1$ through the introduction of $y_1$ as the deviation from $\eta_1$, is due to the fact that the vector-field on $y_0=\eta_1(x_0)$ is not in the tangent space (indicated by the dash-dot lines).}
\figlab{y0defeta1}
\end{figure}
{The SO method is iterative, successively introducing $y_i$'s by considering normal forms of the form
\eqref{xys1} at each step of the iteration. The slow variables $x_0$ will not be transformed during this iteration.} To complete the first step of
the iteration, consider the equation $Y_0(x_0,y_0)=0$, with $Y_0=Y_0(x_0,y_0)$ as in
\eqref{xys1}. This gives, by applying the analytic implicit function theorem, an analytic
solution $y_0=\eta_1(x_0)$ close to 
\begin{eqnarray}
\eta_1(x_0) \approx -A_0(x_0)^{-1}\rho_0(x_0),\eqlab{eta1}
\end{eqnarray}
since $R_0(x_0,y_0)=\mathcal O(y_0^2)$. The graph
$M_1=\{y_0=\eta_1(x_0)\}$ will be an improved slow manifold. To show that
this is indeed an improved slow manifold, one \textit{straightens out} the
new slow manifold by introducing $y_1$ through {$y_0=y_1+\eta_1(x_0)$}. Then the
equations become
\begin{eqnarray*}
 \dot x_0= \epsilon X_1(x_0,y_1),\quad \dot y_1 &=& Y_1(x_0,y_1)=\rho_1(x_0) + A_1(x_0) y_1 + R_1(x_0,y_1),
\end{eqnarray*}
with, much as before, $X_1(x_0,y_1)=X_0(x_0,y_1+\eta_1(x_0))$ and $Y_1(x_0,y_1)=-\epsilon \partial_x \eta_1 X_0(x_0,y_1+\eta_1(x_0))+Y_0(x_0,y_1+\eta_1(x_0))$, and by way of Taylor expanding $Y_1$ about $y_1=0$ we identify the new relevant functions
\begin{eqnarray}
 \rho_1(x_0) &=& -\epsilon \partial_x \eta_1 X_0(x_0,\eta_1(x_0)),\eqlab{rho11}\\
 A_1(x_0)&=&\partial_{y} Y_1(x_0,\eta_1(x_0))=-\epsilon \partial_{x} \eta_1(x_0) \partial_{y} X_0(x_0,\eta_1(x_0)) + \partial_y Y_0(x_0,\eta_1(x_0)),\nonumber
\end{eqnarray}
and $R_1=\mathcal O(y_1^2)$ as the remainder. Therefore formally $\rho_1=\mathcal O(\epsilon^2)$, since $\eta_1 = \mathcal
O(\epsilon)$ according to \eqref{eta1}; an
improvement from $\mathcal O(\epsilon)$ to $\mathcal O(\epsilon^2)$. More accurately:
\begin{eqnarray}
 \Vert \rho_1 \Vert_{C} \le c\epsilon \Vert \rho_0\Vert_{C^1},\quad c>0,\eqlab{rho1est}
\end{eqnarray}
using \eqref{eta1}. Here the subscripts $C$ and $C^1$ are for continuous respectively continuously differentiable functions. 
See \figref{y0defeta1} (b) for a geometrical interpretation. Continuing in this way, at each
step solving
$Y_i(x_0,y_i)=0$ for $y_i=\eta_{i+1}(x_0)$ and then setting
$y_{i}=y_{i+1}+\eta_{i+1}(x_0)$, we obtain an improved error at the end of
each step which is an $\mathcal O(\epsilon)$-multiple of a $C^1$-estimate of the previous error, cf. \eqref{rho1est},
directly leading to the formal $\mathcal O(\epsilon^{n})$-estimates. Making use of the analyticity, that allows for the application of Cauchy estimates to properly estimate the right hand side of \eqref{rho1est},  \cite{kri3} obtained the exponential estimates $\mathcal O(e^{-c/\epsilon})$ using $n=\mathcal O(\epsilon^{-1})$ steps. 

The SO method is identical to the method suggested
by Fraser and Roussel \cite{fra1,fra2,kap3}. This can be
realized by introducing the partial sum $\eta^n=\sum_{i=1}^{n} \eta_i$ and
expanding $Y_{n-1}(x_0,\eta_{n}(x_0))$ as
\begin{eqnarray*}
 Y_{n-1}(x_0,\eta_{n}(x_0)) &=& -\epsilon \partial_{x} \eta_{n-1} X_{n-2}(x_0,\eta_{n-1} +\eta_n) + Y_{n-2}(x_0,\eta_{n-1}+\eta_n)\\
&=&-\epsilon \partial_{x} \eta_{n-1} X_{n-3}(x_0,\eta_{n-2}+\eta_{n-1} +\eta_n) -\epsilon \partial_{x} \eta_{n-2} X_{n-3}(x_0,\eta_{n-2}+\eta_{n-1} +\eta_n)\\
&\quad& + Y_{n-3}(x_0,\eta_{n-2}+\eta_{n-1}+\eta_n)=\cdots\\
&=&-\epsilon \partial_x \eta^{n-1} X_{0}(x_0,\eta^n) + Y_0(x_0,\eta^n). 
\end{eqnarray*}
The equation:
\begin{eqnarray}
-\epsilon \partial_x \eta^{n-1} X_{0}(x_0,\eta^n) +
Y_0(x_0,\eta^n)=0,\eqlab{fraser}
\end{eqnarray}
defines the $n$th step of Fraser and Roussel's iterative method in which one solves for an improved slow manifold $\eta^n$, see \cite{fra1,fra2} and \cite{kap3} where an asymptotic analysis of the method is given. The reference \cite{kap3} does not, however, obtain exponential estimates. Equation \eqref{fraser} with $\eta^{n-1}=\eta^n=\eta$ is the \textit{invariance equation}:
\begin{eqnarray}
 -\epsilon \partial_x \eta X_{0}(x_0,\eta) +
Y_0(x_0,\eta)=0,\eqlab{inveqn}
\end{eqnarray}
for the invariance of $y=\eta(x)$. When the method is viewed within MacKay's setting we can also realize that we can in principle allow $A_0$ to be an unbounded operator: $A_0$ is never measured and therefore it is only necessary to assume that $A_0(x)^{-1}$ is bounded, making the approach potentially useful for partial differential equations. Note that $\rho_1$ vanishes at a true equilibrium where $X_0(x_0,\eta_0)=0$, and the improved slow manifold $M_1=\{y_1=0\}$ therefore includes all equilibria near $M_0$. This property is preserved when using the method iteratively, {and holds regardless of whether \eqref{M0hyp} is satisfied or not; we only need \eqref{M0fast} which is weaker. The importance of this property, as also noted by MacKay \cite{mac1}, is, however, due to the fact that when \eqref{M0hyp} is satisfied then Fenichel's theory, and even the general theory of normally hyperbolic invariant manifolds \cite{murd1,har1,tak1}, guarantees that equilibria are contained within the slow manifold. The general theory further guarantees that the slow manifold contains all nearby invariant sets such as limit cycles and strange attractors. This is obviously not true for the, in general, non-invariant approximation obtained from \eqref{inveqn}.}

{The viewpoint we take in this paper is slightly different from that taken in Fenichel's work. We do not directly connect with $\epsilon=0$. Instead we simply think of $0<\epsilon\ll 1$ as being fixed, much as in \cite{bat1,mac1,nei87}, and describe a procedure how to go from a slow manifold, which is close to being invariant, to one that is even closer to invariance. Indeed, it is clear that we do not need to start the procedure from $\eta_0$. Instead we could start from some guess $\zeta_0$. We would then straighten out this graph by setting $y=\zeta_0(x_0)+y_0$ and write up the corresponding $\rho_0$, $A_0$ and $R_0$. Provided this guess is good enough, i.e. $\rho_0$ is sufficiently small, then the new error will still be of the form $\epsilon \times \Vert \rho_0\Vert_{C^1}$ \eqref{rho1est} and the process can successfully be iterated until exponentially accuracy has been reached. We believe that this viewpoint is appropriate as in many systems an $\epsilon$ is not directly available, see e.g. \cite{templatorref,
bro2}, yet they demonstrate slow-fast behavior. 
Fraser and Roussel have reported convergence problems when the candidate $\zeta_0$ is far from $\eta_0$. Different methodologies have been proposed to overcome this problem \cite{naf1,fra3,dav1}. We re-iterate that one of the key properties of the extension of the SO method we present in this paper, is that it can be formulated in a similar fashion to \eqref{fraser} (see \eqref{phii} below) only involving the vector-field and its first partial derivatives with respect to $x$ and $y$. However, the extension does not share the potential convergence issues related to the SO method as this iterative procedure, involving only linear equations, is successfully initiated by $0$.  }



{The purpose of the following example is three-fold: (i): It demonstrates the use of the method on a system with non-smooth $\epsilon$ dependency. (ii): For this example one can directly compute the SO-generated slow manifold and show that, in accordance with the theory, it is exponentially close to invariance. At the same time one can integrate the equations and show that this is optimal. (iii): It is an example of an ``almost'' invariant normally \textit{elliptic} slow manifold. }
{\begin{example}\exmlab{nex}
This is a modified version of Neishtadt's example, see e.g. \cite{geller1}, on $(x_0,y_0)\in S^1\times \mathbb R^2$:
 \begin{eqnarray*}
  \dot x_0 &=&\epsilon,\\
  \dot y_0 &=& Y_0(x_0,y_0)=\begin{pmatrix}
            \overline{\epsilon} f_{\lfloor \epsilon^{-1} \rfloor}(x_0)\\
            0
           \end{pmatrix}+\begin{pmatrix} 
             0 & 1\\
             -1 &0
            \end{pmatrix}y_0,
 \end{eqnarray*}
 where $$f_{\lfloor \epsilon^{-1} \rfloor}(x_0) = \sum_{k=1}^{\lfloor \epsilon^{-1} \rfloor} e^{-k} \sin (kx_0),\quad \overline{\epsilon} = \epsilon(1+\sin(2\pi \epsilon^{-1})).$$ {Here $S^1=\mathbb R/(2\pi \mathbb Z)$}.
In contrast to Neishtadt, who considered $\overline{\epsilon}=\epsilon$ and the analytic function $f=f_\infty$, we consider the non-smooth versions $\overline{\epsilon}=\overline{\epsilon}(\epsilon)$ and the partial sum $f_{\lfloor \epsilon^{-1} \rfloor}$
 and verify the results of \cite{kri3} on the SO method for a system with a non-smooth dependency on $\epsilon$. The non-smoothness enters through $\sin(2\pi \epsilon^{-1})$ and the greatest integer function: $${\lfloor  \epsilon^{-1} \rfloor }=\max \{N\in \mathbb N\vert N\le \epsilon^{-1}\}.$$ 
 We will not transform the slow variable $x_0$ here so we drop the subscript $0$.

 First we notice that $A_0 =\begin{pmatrix} 
             0 & 1\\
             -1 &0
            \end{pmatrix}$ with eigenvalues $\pm i$. The critical manifold $\{y_0=0\}$ is therefore normally elliptic, not hyperbolic. Also $\rho_0=\begin{pmatrix}
            \overline{\epsilon} f_{\lfloor \epsilon^{-1} \rfloor}(x)\\
            0
           \end{pmatrix}$ which is $\mathcal O(\epsilon)$ since $\Vert \overline{\epsilon}\Vert\le 2\epsilon$ and $\Vert f_{\lfloor\epsilon^{-1}\rfloor}\Vert\le \sum_{k=1}^\infty e^{-k}=\frac{1}{e-1}$.
            
            Applying the SO method once gives $\eta_1$ with
 \begin{eqnarray*}
  \eta_1(x) =  \begin{pmatrix}0\\-\overline{\epsilon} f_{\lfloor\epsilon^{-1}\rfloor}(x)\end{pmatrix}=\mathcal O(\epsilon)
 \end{eqnarray*}
as the solution of $Y_0(x,\eta_1)=0$. 
Next, we set $y_0=y_1+\eta_1(x)$ and 
$$Y_1(x,y_1) = -\epsilon \partial_x \eta_1(x) + Y_0(x,y_1+\eta_1(x))=\begin{pmatrix}
0\\
            \epsilon \overline{\epsilon} f_{\lfloor \epsilon^{-1} \rfloor}'(x)\\
                     \end{pmatrix}+\begin{pmatrix} 
             0 & 1\\
             -1 &0
            \end{pmatrix}y_1.$$
Hence $\rho_1 = \begin{pmatrix}
0\\
            \epsilon \overline{\epsilon} f_{\lfloor \epsilon^{-1} \rfloor}'(x)\\
                     \end{pmatrix}=\mathcal O(\epsilon^2)$ and the solution of $Y_1(x,\eta_2)=0$ is 
 \begin{eqnarray*}
  \eta_2(x) =  \begin{pmatrix}\epsilon \overline{\epsilon} f_{\lfloor \epsilon^{-1} \rfloor}'(x)\\
  0.\end{pmatrix}
 \end{eqnarray*}
Proceeding in this way one can verify the following results:
\begin{eqnarray}
 \eta_n&=& \left\{\begin{array}{cc}
                      \begin{pmatrix}
                       (-1)^{\lfloor n/2\rfloor+1} \epsilon^{n-1} \overline{\epsilon} f^{(n-1)}_{\lfloor\epsilon^{-1}\rfloor}(x)\\
                       0
                      \end{pmatrix}
  & \mbox{if $n$ is even},\\
                                            \\
                        \begin{pmatrix}
                         0\\
                         (-1)^{\lfloor n/2\rfloor+1} \epsilon^{n-1} \overline{\epsilon}f^{(n-1)}_{\lfloor\epsilon^{-1}\rfloor}(x) 
                        \end{pmatrix}
                       & \mbox{if $n$ is odd}
                     \end{array}   \right.= \mathcal O(\epsilon^n) \nonumber
\end{eqnarray}
for $n\ge 3$. The error $\rho_n=-\epsilon \partial_x \eta_n$ is then given as
\begin{eqnarray*}
 \rho_n = \left\{\begin{array}{cc}
                  \begin{pmatrix} 0\\(-1)^{\lfloor n/2\rfloor} \epsilon^{n} \overline{\epsilon} f^{(n)}_{\lfloor\epsilon^{-1}\rfloor}(x)\end{pmatrix}&\mbox{if $n$ is even}\\
                  \\
                  \begin{pmatrix}(-1)^{\lfloor n/2\rfloor} \epsilon^{n} \overline{\epsilon}f^{(n)}_{\lfloor\epsilon^{-1}\rfloor}(x)\\
                   0
                  \end{pmatrix}
&\mbox{if $n$ is odd}
                 \end{array}\right.=\mathcal O(\epsilon^{n+1}).
\end{eqnarray*}
We now estimate $\rho_n$:
\begin{eqnarray*}
 \Vert \rho_n\Vert &\le \epsilon^{n} \Vert \overline{\epsilon}\Vert  \sum_{k=1}^{\lfloor \epsilon^{-1} \rfloor} e^{-k} k^{n}\le 2\epsilon^{n+1} \sum_{k=0}^\infty e^{-k} k^{n}\le 2\epsilon^{n+1} \int_0^\infty e^{-k} k^n dk= 2\epsilon^{n+1} n!
\end{eqnarray*}
We can then apply Stirling's approximation for $n!$ to obtain
\begin{eqnarray*}
 \Vert \rho_n\Vert \le 2e\epsilon^{n+1} n^{n+1/2} e^{-n}.
\end{eqnarray*}
Setting $n= \lfloor\epsilon^{-1}\rfloor$ gives an exponential estimate:
\begin{eqnarray*}
 \Vert \rho_{\lfloor\epsilon^{-1}\rfloor} \Vert \le 2e\epsilon^{\lfloor\epsilon^{-1}\rfloor} \lfloor\epsilon^{-1}\rfloor^{ \lfloor\epsilon^{-1}\rfloor+1/2} e^{- \lfloor\epsilon^{-1}\rfloor}\le 2e\epsilon^{-1/2} e^{-\lfloor\epsilon^{-1}\rfloor}\le 2e^2 \epsilon^{-1/2} e^{-\epsilon^{-1}}.
\end{eqnarray*}
Note that the graph $y_0=\eta^{\lfloor\epsilon^{-1}\rfloor}(x)=\sum_{i=1}^{\lfloor\epsilon^{-1}\rfloor} \eta_i(x)$ is also non-smooth in $\epsilon$. This is in general also the case when the original $\epsilon$-dependency is smooth, as in the classical averaging scenarios considered in e.g. \cite{geller1,nei2,nei3}. Note finally that the result is global in the slow variable $x\in S^1$. 

One cannot improve the estimate beyond an exponential one due to the resonance appearing for $\epsilon^{-1}=N\in \mathbb N$. This can easily be seen by 
introducing the complex variable $p=(y_0)_1+i(y_0)_2$ so that 
\begin{eqnarray*}
  \dot p &=& -ip + N^{-1} f_{N}(x),
  \end{eqnarray*}
 where $\overline{\epsilon}=N^{-1}(1+\sin(2\pi N))=N^{-1}$, and using that $x=N^{-1} t$ after possibly a translation of time.

\end{example}

}


\subsection{Normally hyperbolic slow manifolds and their fibers} 
To explain our extension of the SO method we first need to explain Fenichel's theory a bit further. We will base our discussion on \eqref{xys1} and $M_0=\{y_0=0\}$. For this Fenichel assumed that $X_0$ and $Y_0$ depended smoothly on $\epsilon$ and that $M_0$ satisfies the hyperbolicity condition \eqref{M0hyp}. Then the stable and unstable manifolds persist. To explain this, consider at $\epsilon=0$ the fast fiber $F_0^{z_0}=\{(x_0,y_0)\vert \Vert y_0\Vert \le \Delta \}$, with $\Delta>0$ potentially small because of the localness in $y_0$, based at the point $z_0=(x_0,0)$. If the real parts of the eigenvalues of $A$ are all negative, then $M_0\vert_{\epsilon=0}$ is asymptotically stable for $\epsilon=0$ and all solutions on $F^{z_0}_0$ contract exponentially toward the base point $z_0$. {By Fenichel's theory \cite{fen2}, \cite[Theorem 3, p. 20]{jon1} the fast fibers $F_0^{z_0}$ perturb to $F_\epsilon^{z_0}$.
These different fibers $F_\epsilon^{z_0}$ form a family $\{F_\epsilon^{z_0}\}_{z_0\in M}$ which is invariant in the following sense
\begin{eqnarray*}
 \Phi_0^t(F_\epsilon^{z_0}) \subset F_\epsilon^{\Phi_0^t (z_0) },
\end{eqnarray*}
where $\Phi_0^t$ is the time-$t$ flow map of \eqref{xys1}}. The motion of any point $z=(x,y)\in F_\epsilon^{z_0}$ therefore decomposes into a fast contracting component and a slow component governed by the motion of the base point $z_0$ of the fiber. The assignment $z\mapsto z_0$ is called the fiber projection and we denote it by $\pi_f$. In the physics literature a fiber is also sometimes called an isochron \cite{ajr2}. 

The fiber projection $\pi_f$ is smooth, and so locally there exists a transformation $(u,v)\mapsto (x_0,y_0)$, which is $\epsilon$-close to the identity, mapping \eqref{xys1} into the Fenichel normal form, explained in e.g. \cite[Eq. (3.21), p. 41]{jon1}:
\begin{eqnarray*}
 \dot u &=&\epsilon U(u),\quad
\dot v=V(u,v)v.
\end{eqnarray*}
These are the ideal coordinates for the description of the system near the slow manifold; the slow manifold coincides with the zero level set $\{v=0\}$ and the fibers of the form $F_\epsilon^{(u_b,0)}$, based at $(u,v)=(u_b,0)$, have been straightened out to $\{(u,v)\vert u=u_b,\,\Vert v\Vert \le \Delta\}$. In particular, the matrix $V$ has eigenvalues with purely negative real part. We will approach this ideal by first constructing a transformation $(x,y)\mapsto (x_0,y_0)$ so that the $x$-equation, up to exponentially small error terms, becomes independent of $y$ to linear order:
\begin{eqnarray}
 \dot x & = \epsilon (\Lambda(x)+\mathcal O(y^2))+\mathcal O(e^{-c/\epsilon}),\eqlab{fintp}\\
 \dot y & = A(x)y+\mathcal O(y^2)+\mathcal O(e^{-c/\epsilon}).\nonumber
\end{eqnarray}
 Then the tangent space to the fibers based at $(x,y)=(x_b,0)$ will coincide with $\{(x,y)\vert x=x_b\}$ for all $x_b\in \mathcal U$ (the subscript $b$ is for base) up to exponentially small terms. Later we will also seek to remove the terms that are quadratic in $y$. We will see that we do not need smoothness of $X_0$ and $Y_0$ in $\epsilon$ to construct this transformation. Also we {will} replace the hyperbolicity condition \eqref{M0hyp} by the weaker \textit{fastness condition} \eqref{M0fast}. 

When $M$ is of saddle type, with a stable manifold $W^s(M)$ of dimension $n_f^s$ and an unstable manifold $W^u(M)$ of dimension $n_f^u$ ($n_f=n_f^s+n_f^u$), then Fenichel's normal form takes a slightly different form \cite[Eq. (3.21), p. 41]{jon1}: There exists a transformation $(u,v,w)\mapsto (x_0,y_0)$, with $\text{dim}\,\{v\}={n_f^s}$ and $\text{dim}\,\{w\}={n_f^u}$, which is $\epsilon$-close to the identity, mapping \eqref{xys1} into
\begin{eqnarray}
         \dot u &=&\epsilon (U_0(u)+U_1(u,v,w)vw),\nonumber\\
\dot v&=&V(u,v,w)v,\eqlab{fnf2}\\
\dot w&=&W(u,v,w)w.\nonumber
\end{eqnarray}
Here $U_1(u,v,w):\{v\}\times \{w\}\rightarrow \mathbb R^{n_s}$ is a
bilinear function of $v$ and $w$. The slow manifold is then given by
$\{v=0,\,w=0\}$ with stable manifold $\{w=0\}$ and unstable manifold
$\{v=0\}$. The transformation may only exist in a small neighborhood of the
slow manifold so in general we need $\Vert v\Vert \le \Delta_v$ and $\Vert
w\Vert\le \Delta_w$.

\section{The SOF method}\seclab{sofsec}
In this section we shortly describe our method for approximating the tangent spaces of the fibers. 
Following the
$\mathcal O(\epsilon^{-1})$ applications of the SO method we can start from the real analytic slow-fast system:
\begin{eqnarray*}
 \dot x_0 &=&\epsilon X(x_0,y)=\epsilon (\Lambda(x_0)+\mu_0(x_0)y + T_0(x_0,y)),\\ \dot y& =&Y(x_0,y)=\rho(x_0)+A(x_0)y+R(x_0,y),
\end{eqnarray*}
with $\rho=\mathcal O(e^{-c/\epsilon})$ \cite{kri3} describing the error-field on $\{y=0\}$ and $R=\mathcal
O(y^2)$. Moreover, $\mu_0=\partial_y X(x_0,0)$ and $T$ is the $\mathcal O(y^2)$ remainder from the Taylor expansion of $X$ about $y=0$. For the purpose of
obtaining exponential estimates, we can ignore $\rho$
completely. We shall return to this later. We will assume that there are $n_s$ slow variables $x\in \mathbb
R^{n_s}$ and $n_f$ fast variables $y\in \mathbb R^{n_f}$. The aim is to introduce a succession of transformations of the
form $x_i = x_{i+1}+\epsilon \phi_{i}(x_{i+1})y$ formally pushing the term $\mu_0 y$
in $\epsilon^{-1}\dot x_{0}$ which is linear in $y$ to consecutive
higher orders in $\epsilon$. Let us consider the first step, introducing $x_0=x_1+\epsilon \phi_0(x_1)y$ so that
\begin{eqnarray}
 \dot x_1& =& J^{-1}\left(\epsilon \Lambda+\epsilon \left\{ \epsilon \partial_x \Lambda \phi_0+\mu_0 -\phi_0 A\right\}y+\epsilon \mathcal O(y^2)\right)\nonumber \\
 &=&\epsilon \left(\Lambda+\left\{\epsilon \partial_x \Lambda \phi_0+\mu_0 -\phi_0 A\right\}y-\epsilon\partial_x \phi_0 \Lambda y+\mathcal O(y^2)\right)
\eqlab{xp}
 \end{eqnarray}
 where $J=I_s+\epsilon \partial_{x} \phi_0 y$, $I_s=\text{identity}\in \mathbb R^{n_s\times n_s}$, is the Jacobian of the transformation $x_1\mapsto x_0$, and where we have used the identity
 \begin{eqnarray*}
  J^{-1} = I_s-\epsilon \partial_{x} \phi_0 y+J^{-1}(\epsilon \partial_x \phi_0 y)^2.
 \end{eqnarray*}
 All functions on the right hand side of \eqref{xp} depend on $x_1$, a dependency we for simplicity here have suppressed. 
The term in \eqref{xp} which is linear in $y$ is due to two contributions. The first one is due to the expansion of $X(x_0,y)- \phi Y(x_0,y)$ in $y$, the curly bracket in \eqref{xp}, while the second one:
\begin{eqnarray}
\mu_1 &=&-\epsilon \partial_{x} \phi_0\Lambda,\eqlab{mu1}
\end{eqnarray}
 comes from applying the inverse of the Jacobian. Here $\partial_{x} \phi_0 \Lambda$ is understood column-wise:
\begin{eqnarray*}
 \partial_{x} \phi_0 \Lambda = \left(\partial_{x} (\phi)^1 \Lambda \cdots \partial_{x} (\phi)^{n_f} \Lambda\right),
\end{eqnarray*}
$(\phi)^i=(\phi)^i(x_0)\in \mathbb R^{n_s}$ being the $i$th column of $\phi=\phi(x_0)\in \mathbb R^{n_s\times n_f}$. We let $\phi_0$ be the solution to the linear equation obtained by setting the first contribution, the curly bracket in \eqref{xp}, to zero: 
\begin{eqnarray*}
\epsilon \partial_x \Lambda \phi_0+\mu_0 -\phi_0 A = 0.
\end{eqnarray*}
This equation has a solution $\phi_0$ close to $\mu_0 A^{-1}$, and the new error term $\mu_1$ \eqref{mu1}, which by construction is the only remaining term in $\epsilon^{-1}\dot x_1$ linear in $y$, is therefore formally smaller than the old error $\mu_0$. There is an improvement from $\mathcal O(1)$ to $\mathcal O(\epsilon)$. Note also that
\begin{eqnarray*}
 \Lambda_1 = \Lambda,\quad A_1=A,
\end{eqnarray*}
and in particular $\mu_1$ then vanishes at all equilibria $(x,y)=(x_e,0)$ since there $\Lambda(x_e)=0$. We will use these types of transformations successively in the proof, pushing the error term to higher order in $\epsilon$. One of the main results of the paper is that eventually the error is exponentially small: $\mu = \mathcal O(e^{-c/\epsilon})$. {Again we stress that the system is assumed to be analytic.} We present the first result formally in \thmref{thm1} which we prove in \secref{prvthm1}. In \secref{curve} we present a result, \thmref{thm2}, on approximation of the curvature of the fibers. \thmref{thm2} excludes normally elliptic slow manifolds and neutral saddle-type slow manifolds where both $\lambda$ and $-\lambda$, $\text{Re}\,\lambda\ne 0$, are eigenvalues of $A$. This requirement appears in the construction of the appropriate transformations, where we encounter linear matrix equations of the form:
\begin{eqnarray*}
 A^T \psi^i + \psi^i A=Q^i,
\end{eqnarray*}
for the unknown matrices $\psi^i$. Solutions of this linear problem
exist and are unique if and only if $\sigma(A)\cap \sigma(-A) =
\emptyset$, see \cite[Theorem 4.4.6]{hor1}. The case where both $\lambda$ and $-\lambda$ are eigenvalues of the $A$ leads to small
divisors, as in the problem of analytic linearization \cite{guc1}. We should mention that small divisors are not necessarily an immovable obstruction, see e.g. \cite{yoc1} for the problem of analytic linearization, \cite{pos2} for KAM theory, and \cite[p.26]{chi1} for the Hartman-Grobman theorem. Still, trying to remove such an obstruction, is not within the scope of this work.

\section{Main results}\seclab{main}
We consider the real analytic slow-fast system \eqref{xys1} in the form
 \begin{eqnarray}
 \dot x_0 &=&\epsilon X_0(x_0,y_0)=\epsilon (\Lambda_0(x_0)+\mu_0(x_0)y_0 + T_0(x_0,y_0)),\eqlab{xys12}\\ 
\dot y_0 &=&Y_0(x_0,y_0)=\rho_0(x_0)+A_0(x_0)y_0+R_0(x_0,y_0), \nonumber\\
R_0&&\hspace{-.5cm}(x_0,y_0),\,T_0(x_0,y_0) = \mathcal O(y_0^2),\nonumber
\end{eqnarray}
with $n_s$ slow variables and $n_f$ fast ones so that $x_0\in \mathcal U+i\chi_0 \subset \mathcal X_{\mathbb C} = \mathbb C^{n_s}$ and $y_0\in \mathcal V+i\nu_0\subset \mathcal Y_{\mathbb C} = \mathbb C^{n_f}$. Here $\mathcal U\subset \mathcal X=\mathbb R^{n_s}$ and $\mathcal V\subset \mathcal Y=\mathbb R^{n_f}$ are real open subsets. 
\begin{theorem}\thmlab{thm1}
  Fix $0\le \underline{\chi}<\chi_0$ and $0\le \underline{\nu}<\nu_0$. Then there exists an $\epsilon_0>0$ so that for all $\epsilon\le \epsilon_0$ \textnormal{the SOF method} constructs a transformation $(x,y)\mapsto (x_0,y_0)$ which is $\epsilon$-close to the identity from $(\mathcal U+i\underline{\chi})\times (\mathcal V+i\underline{\nu})$ to $(\mathcal U+i{\chi}_0)\times (\mathcal V+i{\nu}_0)$ mapping \eqref{xys12} into
\begin{eqnarray}
 \dot x&=&\epsilon (\Lambda(x)+\mu(x)y + Q(x)y^2 +C(x,y)),\eqlab{xys13} 
\\
\dot y &=&\rho(x)+A(x)y+R(x,y),  \nonumber
\end{eqnarray}
with $\mu$ and $\rho$ vanishing at equilibria $(x_e,y_e)$ where
\begin{eqnarray}
\Lambda(x_e)=0,\quad y_e=0,\eqlab{eqset}
\end{eqnarray}
and both $\mu$ and $\rho$ are exponentially small
\begin{eqnarray*}
 \gamma= \Vert \mu \Vert_{\underline{\chi}},\,\delta = \Vert \rho \Vert_{\underline{\chi}} = \mathcal O(e^{-c_1/\epsilon}),
\end{eqnarray*}
$T(x,y)=Q(x)y^2 +C(x,y)$, $C=\mathcal O(y^3)$, and
 \begin{eqnarray*}
 \Vert \Lambda-\Lambda_0\Vert_{\underline{\chi}},\,\Vert A-A_0\Vert_{\underline{\chi}},\,\Vert T-T_0\Vert_{\underline{\chi},\underline{\nu}},\,\Vert R-R_0\Vert_{\underline{\chi},\underline{\nu}} &\le c_2 \epsilon,
\end{eqnarray*}
for some constants $c_1$ and $c_2$.
\end{theorem}

We highlight that the estimates are not uniform in $\underline{\chi}$ and $\underline{\nu}$. We also have the following corollary which provides a convenient form for the transformation in \thmref{thm1}:
\begin{corollary}\corlab{cornum}
If the eigenvalues of $A_0$ all have non-zero real part and {the dependency of $\epsilon$ is smooth} then there exists an $\epsilon_0>0$ so that for all $\epsilon\le \epsilon_0$ there exist a slow manifold $M$ of \eqref{xys1} and $N_1,\,N_2=\mathcal O(\epsilon^{-1})\in \mathbb N$ so that $M$ is given as the graph
\begin{eqnarray}
 y_0 = \eta(x_0)+\mathcal O(e^{-c_1/\epsilon}),\eqlab{y0sm}
\end{eqnarray}
with 
\begin{eqnarray}
\eta =\sum_{n=1}^{N_1} \eta_n=\mathcal O(\epsilon),\eqlab{eta}
\end{eqnarray}
where the partial sums $\eta^n
\equiv \sum_{i=1}^n \eta_i$ satisfy \eqref{fraser}, repeated here for convenience,
\begin{eqnarray}
-\epsilon \partial_x \eta^{n-1} X_{0}(x_0,\eta^n) +
Y_0(x_0,\eta^n)=0,\eqlab{fraser2}
\end{eqnarray}
for $1\le n\le N_1$ using the convention $\eta^{0}\equiv 0$.
Furthermore, the tangent space of the fibers $\mathcal F_\epsilon^{z_0}$ at
the base point $z_0=(x_0,y_0)$, with $y_0$ as in \eqref{y0sm}, is given as
\begin{eqnarray}
 T_{z_0} \mathcal F^{z_0}_\epsilon = \textnormal{Rg}\,\left(
\begin{pmatrix}
\epsilon {\phi}(x_0)\\
I_f+\epsilon \partial_x {\eta}(x_0){\phi}(x_0)
\end{pmatrix}+\mathcal O(e^{-c_2/\epsilon})\right).\eqlab{tsp}
\end{eqnarray}
Here $I_f=\text{identity}\in \mathbb R^{n_f\times
n_f}$ and ${\phi}=\sum_{n=0}^{N_2}\phi_n$ where 
the partial sums 
\begin{eqnarray*}
\phi^n\equiv\sum_{i=0}^n\phi_i,
\end{eqnarray*}
satisfy
\begin{eqnarray}
  \epsilon (\partial_x X_0+\partial_y X_0 \partial_{x} {\eta})\phi^{n} -\epsilon \partial_x \phi^{n-1} X_0+\partial_y X_0 - \phi^{n} (-\epsilon \partial_{x} {\eta} \partial_y X_0+\partial_y Y_0) = 0,\eqlab{phii}
\end{eqnarray}
for $0\le n\le N_2$ using the convention $\phi^{-1}\equiv 0$. The functions $X_0,\,\partial_x X_0,\,\partial_y X_0$ and $\partial_y Y_0$ in \eqref{phii} are all evaluated at $(x_0,{\eta}(x_0))$, 
\end{corollary}

\begin{proof}
  Here Fenichel's theorems applies \cite[Theorem 2 and 3, pp. 8 and 20]{jon1}. {The existence of $\eta$ in \eqref{y0sm} is guaranteed by \cite[Lemma 1]{nei87}. Their $\eta$ is, however, generated by appropriate linearizations of \eqref{fraser2}. The proof can, nevertheless, be modified so that the updates are based on \eqref{fraser2}. This is done in \cite{kri3}.}
For the second part, note that each $\phi_n$ solves \eqref{phieqni} below.
Here $\Lambda(x_0) = X_0(x_0,\eta(x_0))$ and $A(x_0) = -\epsilon \partial_{x} {\eta}(x_0) \partial_y  X_0(x_0,\eta(x_0))+\partial_y Y_0(x_0,\eta(x_0))$ are given through the already determined $\eta$, and \eqref{phii} then follows by summation over $n$. Also since the method generates a transformation $(x,y)\mapsto (x_0,y_0)$ of the form 
\begin{eqnarray}
 x_0&=&x+\epsilon {\phi}(x)y + \mathcal O(y^2),\eqlab{x02x}\\
 y_0 &=& y+{\eta}(x_0).\nonumber
\end{eqnarray}
we obtain a tangent vector to the curve $\theta=\theta((y)_i)$ at $(x_0,\eta(x_0))$ as
\begin{eqnarray*}
 \theta'(0) = \begin{pmatrix}
\epsilon ({\phi})^i\\
e_i+\epsilon \partial_{x} {\eta} ({\phi})^i
\end{pmatrix}.
\end{eqnarray*}
Here $(e_i)_j=\delta_{ij}$ Kronecker's delta, and
$(\phi)^i=(\phi)^i(x_0)\in \mathbb R^{n_s}$ is the $i$th column of
$\phi=\phi(x)\in \mathbb R^{n_s\times n_f}$. 
\end{proof}
 {\begin{remark}
  Note that we have assumed in \corref{cornum} that the dependency on $\epsilon$ is smooth. We need this to be able to invoke Fenichel's theory. However, \thmref{thm1} still applies for non-smooth $\epsilon$-dependency and non-hyperbolic slow manifolds $M_0$ satisfying \eqref{M0fast}. In this case the corollary just presents a convenient form \eqref{phii} of the SOF method only involving the vector-field and its first partial derivatives. Moreover, the tangent spaces should in this case be interpreted not by Fenichel's theory but by the fact that they lead to \eqref{xys13}. 
 \end{remark}}
 {\begin{remark}
      As highlighted in \eqref{tsp} the $n_f$ columns of the matrix
\begin{eqnarray}
 \begin{pmatrix}
  \epsilon \phi(x_0)\\
  I_f + \epsilon \partial_x \eta(x_0) \phi(x_0),
 \end{pmatrix}\in \mathbb R^{n\times n_f},\eqlab{matrix1}
\end{eqnarray}
span the tangent space to the fibers based at $(x_0,\eta(x_0))$. Consequently, the $n_s$ rows of 
\begin{eqnarray}
 \begin{pmatrix}
  I_s+\epsilon \phi(x_0)\partial_x \eta(x_0)\,\,-\epsilon \phi(x_0)
 \end{pmatrix}\in \mathbb R^{n_s\times n},\eqlab{matrix2}
\end{eqnarray}
span the normal space to the fibers. Indeed, \eqref{matrix2} has rank $n_s$, which is the dimension of the normal space, and if we multiply \eqref{matrix2} on the left of \eqref{matrix1} we obtain
\begin{eqnarray*}
 \begin{pmatrix}
  I_s+\epsilon \phi(x_0)\partial_x \eta(x_0)\,\,-\epsilon \phi(x_0)
 \end{pmatrix} \begin{pmatrix}
  \epsilon \phi(x_0)\\
  I_f + \epsilon \partial_x \eta(x_0) \phi(x_0),
 \end{pmatrix} 
 &=& 0\in \R^{n_s\times n_f}.
\end{eqnarray*}
     \end{remark}}

We believe that these results, in particular in the form (\eqsref{fraser2}{phii}) presented in Corollary 1, are useful in computations as the approximation of the relevant objects, the slow manifold and its tangent spaces, only require evaluations of the initial vector-field and its gradients. In particular, we believe that the approximations of the tangent spaces can be usefully applied in examples with many fast degrees of freedoms where one is faced with having to trade off accuracy with minimizing computational effort. {From a given point $(x_0,y_0)$}, near the slow manifold $M=\{y_0=\eta(x_0)\}$ \eqref{eta}, one can approximate the fiber projection $\pi_f:(x_0,y_0)\mapsto (x_b,\eta(x_b))$, onto the base point, by solving the equations 
\begin{eqnarray}
x_0 &=&x_b^{\text{app}}+\epsilon {\phi}(x_b^{\text{app}})y,\eqlab{xbapp}\\
y_0 &=& y+\eta(x_0),\nonumber
\end{eqnarray}
for $y$ and $x_b^{\text{app}}$.  The second equation gives $y=y_0-\eta(x_0)$
which inserted into the first equation gives a non-linear equation for
$x_b^{\text{app}}$. The right hand side of this equation is, however,
$\epsilon$-close to the identity. {Eq. \eqref{xbapp} is similar to Eq. (3.9) in \cite{rob89} ($x^{\mathcal M}$ in \cite{rob89} playing the role of $x_b^{\text{app}}$, $h$ the role of $\eta$ and $P$ the role of $\epsilon \phi$).} 

In \cite{kap4} it is stated that this
projection is only $\mathcal O(\epsilon)$, and therefore asymptotically in
$\epsilon$ not better than the ``naive projection'' $(x_0,y+\eta(x_0))
\mapsto (x_0,\eta(x_0))$. However, this estimate is for fixed $y$. We
believe it is more appropriate to highlight that the error is of the form:
\begin{eqnarray*}
 \Vert \pi_f(x_0,y_0)-(x_b^{\text{app}},\eta(x_b^{\text{app}}))\Vert =\mathcal O(\epsilon y^2),
\end{eqnarray*}
ignoring here the exponentially small error terms. 
 It is exact up the exponentially small error terms if
 the tangent space is a hyperplane (which \cite{kap4} also highlights). The
different projections are illustrated in \figref{proj}. The linear
projection accounts for the ``initial slip'' \cite{ajr2} along the slow
manifold.

\begin{figure}[h!]
\begin{center}
{\includegraphics[width=.75\textwidth]{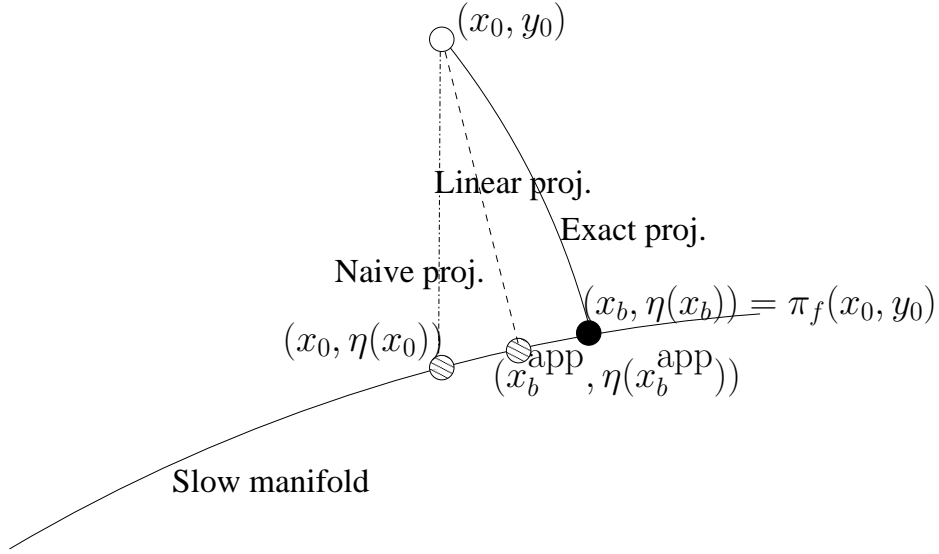}}
\end{center}
\caption{Illustration of the different projections: naive, linear and exact. The linear projection $(x_0,y_0)\mapsto (x_b^{\text{app}},\eta(x_b^{\text{app}}))$ is given by the equations in \eqref{xbapp}.  }
\figlab{proj}
\end{figure}

By approximating the fiber projection we can compute approximations to the dynamics having only to propagate initial conditions on $\mathcal O(1)$ time scales, splitting the problem into first propagating the base point $x_b=x_b(\tau)$, $\tau=\epsilon t$, through
\begin{eqnarray*}
x_b' &=& X_0(x_b,\eta(x_b)),\\
()'&=&\frac{d}{d\tau},
\end{eqnarray*}
and then follow this by propagating $y_0=y_0(t)$ through
\begin{eqnarray*}
 \dot y_0 &=&Y_0(x_0,y_0+\eta(x_0)),
\end{eqnarray*}
using $x_0=x_{b}+\epsilon \phi(x_b)y$ and the solution $x_b=x_b(\tau)$ obtained from the first step. This is the subject of our next paper \cite{kri4}.


\section{Proof of \thmref{thm1}}\seclab{prvthm1}
We first make use of the SO method and the result from \cite{kri3} to determine $\eta$ and transform via $y_0=y+\eta(x_0)$ \eqref{xys12} into
\begin{eqnarray}
 \dot x_0 &=&\epsilon X(x_0,y)=\epsilon (\Lambda(x_0)+\mu_0(x_0)y+T_0(x_0,y)),\eqlab{xys22}\\
 \dot y &=&Y(x_0,y)=\rho(x_0)+A(x_0)y+R(x_0,y),\nonumber
\end{eqnarray}
defined on the domain $(x_0,y)\in (\mathcal U+i\chi)\times (\mathcal V+i\nu)$ with $\chi = (\underline{\chi}+\chi_0)/2$ and $\nu =\underline{\nu}$, and where $\rho=\mathcal O(e^{-C_1/\epsilon})$. We initially ignore this term setting $\rho\equiv 0$. Furthermore, $X$ and $Y$ are $\epsilon$-close to $X_0$ respectively $Y_0$ being given by
\begin{eqnarray}
 X(x_0,y) &=& X_0(x_0,\eta(x_0)+y),\eqlab{XY}\\ 
Y(x_0,y) &=& -\epsilon \partial_x \eta X_0(x_0,\eta(x_0)+y) + Y_0(x_0,\eta(x_0)+y),\nonumber
\end{eqnarray}
so that also
\begin{eqnarray}
 \Lambda(x_0) &=& X(x_0,0),\quad \mu_0(x_0) = \partial_y X(x_0,0),\quad  A(x_0) =\partial_y Y(x_0,0).\eqlab{A}
\end{eqnarray}
The functions $R$ and $T_0$ are the quadratic remainders from the Taylor expansion of $Y$ respectively $X$ about $y=0$.
 Let $K$, $C_\Lambda$ and $C_\Lambda^\prime$ be so that $\Vert A^{-1}\Vert_{\chi} \le \frac{K}{2}$, $\Vert \Lambda\Vert_{\chi}\le C_\Lambda$ and $\Vert \partial_x \Lambda \Vert_{\chi}\le C_\Lambda^\prime$. 
 
 We define the error $\gamma_0$ by
\begin{eqnarray*}
 \gamma_0 = \Vert \mu_0\Vert_{\chi},
\end{eqnarray*}
and apply the transformation
\begin{eqnarray*}
x_0 = x_1+\epsilon \phi_0 (x_1) y,
\end{eqnarray*}
where $\phi_0$ solves 
\begin{eqnarray}
\epsilon \partial_x \Lambda \phi_0+\mu_0 -\phi_0 A = 0.\eqlab{phieqn}
\end{eqnarray}
Cf. \eqref{xp} this transforms the system into
\begin{eqnarray*}
 \dot x_1 &=& \epsilon (\Lambda(x_1)+\mu_1(x_1) y+T_1(x_1,y))\\
\dot y &=&A(x_1)y+R_1(x_1,y),
\end{eqnarray*}
with
\begin{eqnarray*}
 \mu_1 = -\epsilon \partial_{x} \phi_0 \Lambda.
\end{eqnarray*}
\begin{remark}\remlab{mu1prop}
The new error function $\mu_1$ vanishes at an equilibrium of \eqref{xys22} where
$\Lambda = 0$. This implies that the linearization in these coordinates
takes a very suitable form with the linearized slow dynamics
$\frac{d\delta x_1}{dt} = \epsilon \partial_x \Lambda \delta x_1$ exactly
independent of the fast variables. This property is preserved during the iteration.
\end{remark}
Note that
\begin{eqnarray}
 \Vert x_1-x_0\Vert_{\chi,\nu}= \epsilon \Vert \phi_0 y\Vert_{\chi,\nu} \le
\epsilon \gamma_0 \sigma,\eqlab{dx}
\end{eqnarray}
where $\sigma = \sup_{y\in \mathcal V+i\underline{\nu}} \Vert y\Vert<\infty$. From the linear equation \eqref{phieqn} we immediately obtain the following Lemma.
\begin{lemma}
If $\epsilon \le 1/(K C_\Lambda^\prime)$ then the solution of \eqref{phieqn} satisfies
\begin{eqnarray}
 \Vert \phi_0 \Vert_{\chi}\le K\gamma_0.\eqlab{phi0}
\end{eqnarray}
\end{lemma}
\begin{proof}
 Take $\phi_0^0 = \mu A^{-1}$, let $r= \frac{K}{2}\gamma$ and introduce
$\phi_0 = \phi_0^0+z$ so that \eqref{phieqn} becomes
\begin{eqnarray*}
 z = F(z),
\end{eqnarray*}
where $F(z) = \epsilon \partial_x \Lambda (\phi_0^0 +z) A^{-1}$. We have
\begin{eqnarray*}
 \Vert F(z)\Vert_{\chi} \le { \epsilon K C_\Lambda^\prime} r \le r,\\
\Vert \partial_{z} F\Vert_{\chi} \le \frac{\epsilon K
C_\Lambda^\prime}{2}<1.
\end{eqnarray*}
Here we have used the assumption $\epsilon\le 1/(K C_\Lambda^\prime)$. The function $F$ is therefore a contraction on $B_r\subset \mathcal V+i\nu$ and there exists a unique solution of \eqref{phieqn} with
\begin{eqnarray*}
 \Vert \phi_0\Vert_{\chi}\le 2r = K\gamma_0.
\end{eqnarray*}
\end{proof}

The solution is also analytic in $x_0$. Using this lemma we can then estimate the new error using a Cauchy
estimate
\begin{eqnarray*}
 \gamma_1\equiv \Vert \mu_1 \Vert_{\chi_1} \le \epsilon
\frac{KC_{\Lambda}}{\xi_0} \gamma_0,
\end{eqnarray*}
where $\chi_1 = \chi-\xi_0$. 

We now use this
result successively, introducing $x_n = x_{n+1}+\epsilon \phi_n(x_{n+1})y$
with $\phi_n$ solving
\begin{eqnarray}
\epsilon \partial_x \Lambda\phi_{n} +\mu_{n} - \phi_{n} A = 0,\quad 
\mu_n = -\epsilon \partial_{x} \phi_{n-1} \Lambda,\eqlab{phieqni} 
\end{eqnarray}
on $x\in \mathcal U+i\chi_n$, $\chi_n=\chi-\sum_{i=0}^{n-1} \xi_n$, for
each $n\ge 1$. We take $\xi_n=\overline{\xi}=2\epsilon KC_{\Lambda}$ at each step so
that
\begin{eqnarray}
 \gamma_{n+1} \le \epsilon \frac{KC_{\Lambda}}{\xi_n}\gamma_{n}\le
\frac{1}{2}\gamma_{n}\le 2^{-(n+1)} \gamma_0,\eqlab{gamman}
\end{eqnarray}
 with $\gamma_{n} = \Vert \mu_n\Vert_{\chi_{n}}$, $\chi_{n} = \chi-n\overline{\xi}$.  Note also that
\begin{eqnarray*}
 \Vert x_{n}-x_0\Vert_{\chi_n,\nu} &\le \sum_{i=0}^{n-1} \Vert x_{i+1} - x_{i}\Vert_{\chi_n,\nu}\le \epsilon \sigma \sum_{i=0}^{n-1} 2^{-i} \gamma_0\le 2\epsilon \sigma \gamma_0.
\end{eqnarray*}
Setting $\chi_{N_2} = \underline \chi$ we realize that we can take
$N_2=\frac{\chi-\underline{\chi}}{\overline{\xi}}=\mathcal O(\epsilon^{-1})$ steps
so that
\begin{eqnarray*}
 \gamma_{N_2} \le
2^{-\left(\frac{\chi_0-\underline{\chi}}{4\epsilon K
C_{\Lambda}}\right)} \gamma_0.
\end{eqnarray*}
{Now, let $\phi = \sum_{n=1}^{N_2} \phi_n$ and notice that the difference between $x_{N_2}\mapsto x_0$ and 
\begin{eqnarray}
x\mapsto x_0=x+\epsilon \phi(x) y,\eqlab{x02x2}
\end{eqnarray}
is $\mathcal O(y^2)$, see also \eqref{x02x}. Therefore applying \eqref{x02x2} to \eqref{xys22} gives
\begin{eqnarray*}
 \dot x &=&\epsilon \left(\Lambda(x) +\left(\mu_{N_2} +\epsilon \partial_x \phi (\phi(x) \rho) \right)y +\mathcal O(y^2)\right).
\end{eqnarray*}
The term in $\epsilon^{-1}\dot x$ which is linear in $y$ is exponentially small, and the result therefore follows. }

\begin{remark}\remlab{sofn}
Here we consider a fixed number applications of the SOF method, and show
that the extension should only be iterated as many times as the first part
has been iterated for the approximation of the slow manifold. To show this,
we fix $k\in \mathbb N_0$, and apply the SO method $k$ times to
\eqref{xys1} so that the equations for $(x_0,y_k)=(x_0,y_0-\eta^k(x_0))$
are
\begin{eqnarray*}
 \dot x_0 &=&\epsilon X_0(x_0,y_k+\eta^k),\\
\dot y_k &=&\rho_k(x_0) + A_k(x_0)y_k + R_k(x_0,y_k),
\end{eqnarray*}
with $\rho_k=\mathcal O(\epsilon^{k+1})$. (The SO method has then
been applied $k+1$ times to the original equations \eqref{xys0}.) Next, we
apply the SOF method to these equations by introducing the transformation
$x_0=x_{n+1}+\epsilon \phi^{n}(x_{n+1})y_k$ where $\phi^{n}=\mathcal O(1)$
solves \eqref{phii} with $\eta$ replaced by $\eta^k$:
\begin{eqnarray*}
 \dot x_{n+1} &=& \epsilon X_{n+1}(x_{n+1},y_k)\equiv \epsilon \bigg(X_0-\phi^{n}\rho_n\\
&\quad&+ \bigg\{\epsilon (\partial_x X_0+\partial_y X_0 \partial_{x} \eta^k)\phi^{n} -\epsilon \partial_x \phi^{n-1} X_0+\partial_y X_0 - \phi^{n} (-\epsilon \partial_{x} {\eta}^k \partial_y X_0+\partial_y Y_0)\bigg\}y_k \\
&\quad &-\epsilon \bigg(\partial_x (\phi^{n}-\phi^{n-1}) X_0+ \phi^{n}\partial_x \rho_k \phi^n\bigg)y_k+\mathcal O(y_k^2)\bigg)\\
&=&\epsilon \bigg(X_0-\phi^{n}\rho_n+\bigg(\epsilon \partial_x \phi_{n} X_0+ \epsilon \phi^{n}\partial_x \rho_k \phi^{n}\bigg)y_k+\mathcal O(y_k^2)\bigg).
\end{eqnarray*}
Here $\phi_{n} = \phi^{n}-\phi^{n-1} = \mathcal O(\epsilon^{n})$ cf. \eqsref{phi0}{gamman}, replacing the subscripts $0$ with $n$'s in \eqref{phi0}. The functions
$X_0$, $Y_0$ and their derivatives are all evaluated at
$(x_0,\eta^k(x_0))$. Therefore the error, that is the term in $X_{n+1}$
linear in $y_k$ is formally of order 
\begin{eqnarray*}
\epsilon \partial_x \phi_{n} X_0+ \epsilon \phi^{n}\partial_x \rho_k \phi^{n} = \mathcal O(\epsilon^{n+1})+\mathcal O(\epsilon^{k+1})=\mathcal O(\epsilon^{\min({n},k)+1}),
\end{eqnarray*}
and, as expected, there is no improvement for $n$ beyond $k$. A similar result holds true when considering the transformations in \secref{curve} that seek to remove the terms in the slow vector field that are quadratic in the fast variables. 
\end{remark}

\section{Michaelis-Menten-Henri model}\seclab{appl}
In this section we demonstrate our method on the Michaelis-Menten-Henri model 
\begin{eqnarray}
 \dot x &=& \epsilon X(x,y)=\epsilon (-x+(x+\kappa-\lambda)y),\eqlab{mmh}\\
\dot y &=&Y(x,y) = x-(x+\kappa)y,\nonumber
\end{eqnarray}
for enzyme kinetics \cite{kri3}. Here $x$ and $y$ are non-negative concentrations and the parameters satisfy $\kappa>\lambda>0$ and $0<\epsilon\ll 1$. Setting $Y(x,y)=0$ gives $y=\eta_0(x)=\frac{x}{x+\kappa}$ and so $(x,y)=(x_0,y_0+\eta_0(x_0))$ transforms the system into
\begin{eqnarray}
 \dot x_0 &=& \epsilon X_0(x_0,y_0)=\epsilon \left(-\frac{\lambda x_0}{x_0+\kappa }+(x_0+\kappa-\lambda)y_0\right),\eqlab{eX0Y0mmh}\\
\dot y_0 &=&Y_0(x_0,y_0) = \frac{\kappa \lambda x_0}{(x_0+\kappa)^3}\epsilon -\left(x_0+\kappa+\frac{\kappa (x_0+\kappa-\lambda)}{(x_0+\kappa)^2}\epsilon \right)y_0,\nonumber
\end{eqnarray}
 Therefore if $\kappa\gg \epsilon$, so that $$A_0=\partial_{y} Y_0(x_0,0)\equiv x_0+\kappa+\frac{\epsilon \kappa (x_0+\kappa-\lambda)}{(x_0+\kappa)^2} \gg\epsilon,$$ $x_0$ being non-negative, then the system is slow-fast with $n_s=1$ and $n_f=1$. The variable $x_0$ is slow with $\dot x_0=\mathcal O(\epsilon)$ and $y_0$ is fast with $\vert A_0(x_0)^{-1}\vert \ll \epsilon^{-1}$ for $\epsilon \ll 1$.

\subsection{Analytic expressions of $\eta$ and $\phi$ to $2$nd order}
We now obtain analytic expressions for $\eta$ and $\phi$. {We compare our result with \cite{kap4}, where the CSP method was applied to the same model, at the end of this section}. Since the model is linear in the fast variable the SOF method only involves the solution of linear equations. First, we introduce $\eta_1$ satisfying $Y_0(x_0,\eta_1)=0$:
\begin{eqnarray*}
 \eta_1(x_0) = \frac{\kappa \lambda x_0}{(x_0+\kappa)((x_0+\kappa)^3+\epsilon \kappa(x_0+\kappa-\lambda))}\epsilon.
\end{eqnarray*}
Then we define $Y_1(x_0,y_1)=-\epsilon \partial_x \eta_0 X_0(x_0,\eta_1+y_1)+Y_0(x_0,\eta_1+y_1)$ and determine $\eta_2$ from the condition $Y_1(x_0,\eta_2)=0$. We obtain
\begin{eqnarray*}
 \eta_2 = {\frac { x_0 \left( \kappa-3\,x_0 \right) {\lambda}^{2}\kappa}{ \left( x_0+
\kappa \right) ^{7}}}{\epsilon}^{2}+\mathcal O \left( {\epsilon}^{3}
\right),
\end{eqnarray*}
and therefore
\begin{eqnarray*}
y_0=\eta^2 = \eta_1+\eta_2={\frac {\kappa\,\lambda x_0}{ \left( x_0+\kappa
 \right) ^{4}}}\epsilon-{\frac {\kappa \lambda x_0 \left( \kappa\, \left( \kappa-2\,\lambda
 \right) + \left( \kappa+3\,\lambda \right) x_0 \right) }{ \left( x_0+
\kappa \right) ^{7}}}
{\epsilon}^{2}+\mathcal O \left( {\epsilon}^{3} \right),
\end{eqnarray*}
 as a second order approximation of the slow manifold. The error-field is
\begin{eqnarray}
 \rho(x_0) = Y_2(x_0,0) = { \frac {{\lambda}^{3}\kappa x_0\left( {\kappa}^{2}-12 \kappa x_0+15\,{x_0}^{2} \right) }{
 \left( x_0+\kappa \right) ^{9}}}{\epsilon}^{3} +\mathcal O \left(
{\epsilon}^{4}
 \right).\eqlab{rmmh}
\end{eqnarray}
To approximate the fiber directions we introduce $y$ through $y_0=\eta^2+y$ and compute
\begin{eqnarray}
 \Lambda &=&X_0(x_0,\eta^2(x_0))=-{\frac {\lambda\,x_0}{\kappa+x_0}}+{\frac { \left( \kappa-\lambda+x_0
 \right) \kappa\,\lambda\,x_0}{ \left( \kappa+x_0 \right) ^{4}}}\epsilon\nonumber\\
&\quad &-{
\frac { \left( \kappa-\lambda+x_0 \right) \kappa\,\lambda\,x_0 \left( {
\kappa}^{2}-2\,\kappa\,\lambda+(\kappa+3\,\lambda)x_0 \right) }{
 \left( \kappa+x_0 \right) ^{7}}}{\epsilon}^{2}+\mathcal O \left(
{\epsilon}^{3}
 \right)\eqlab{Lmmh}\\
A &=& \partial_y Y_0(x_0,\eta^2(x_0)) = -\kappa-x_0-{\frac {\kappa\, \left( \kappa-\lambda+x_0 \right) }{ \left( 
\kappa+x_0 \right) ^{2}}}\epsilon-{\frac { \left( \kappa-3\,x_0 \right) 
 \left( \kappa-\lambda+x_0 \right) \kappa\,\lambda}{ \left( \kappa+x_0
 \right) ^{5}}}{\epsilon}^{2}\nonumber\\
 &\quad &+\mathcal O \left( {\epsilon}^{3}
\right),\eqlab{Ammh}\\
\mu_0 &=& x_0+\kappa-\lambda.\nonumber
\end{eqnarray}
Here $y$ should be $y_2$, $\rho=\rho_2$, $\Lambda=\Lambda_2$ and $A=A_2$ but we prefer to drop the subscript so that we are in the position of \eqref{xys22} and can refer to \eqref{phieqni}. 
Inserting these expressions into \eqref{phieqni} with $n=0$ gives
\begin{eqnarray*}
\phi_0 &=&\frac{\mu_0}{A-\epsilon \partial_x \Lambda}=-{\frac {x_{{0}}+\kappa-\lambda}{
x_{{0}}+\kappa}} +{\frac { \left( x_{{0}}+\kappa-\lambda
 \right) \kappa\, \left( x_{{0}}\kappa-2\lambda\right) }{
 \left( x_{{0}}+\kappa \right) ^{4}}}\epsilon 
+\mathcal O(\epsilon^2).
\end{eqnarray*}
At the next step, we first compute the new error
\begin{eqnarray*}
\mu_1 &=& -\epsilon \partial_x \phi_0  \Lambda= -{\frac {{\lambda}^{2}x_0}{ \left( \kappa+x_0 \right) ^{3}}}\epsilon+O
 \left( {\epsilon}^{2} \right),
\end{eqnarray*}
and via \eqref{phieqni} with $n=1$ we solve for $\phi_1$ 
\begin{eqnarray*}
 \phi_1 &=&\frac{\mu_1}{A-\epsilon \partial_x \Lambda }={\frac {{\lambda}^{2}x_0}{ \left( \kappa+x_0 \right) ^{4}}}\epsilon+\mathcal O(\epsilon^2).
\end{eqnarray*}
Then 
\begin{eqnarray*}
\mu_2 &=&{\frac {x_{{0}} \left( \kappa-3\,x_0 \right) {\lambda}^{3}}{ \left( \kappa +x_{{0}}\right)^6}}{\epsilon}^{2}
+\mathcal O(\epsilon^3),
\end{eqnarray*}
 so that $\phi_2$ via \eqref{phieqni} with $n=1$ becomes:
\begin{eqnarray*}
 \phi_2 &=& -{\frac {x_{{0}} \left( \kappa-3\,x_{{0}} \right) {\lambda}^{3}}{ \left( x_{{0}}+\kappa \right) ^{7}}}{
\epsilon}^{2}+\mathcal O(\epsilon^3).
\end{eqnarray*}
Let ${\phi}^2=\phi_0+\phi_1+\phi_2$:
\begin{eqnarray*}
 \phi^2 &=&-{\frac {x_{{0}}+\kappa-\lambda}{x_{{0}}+\kappa}}+{\frac { \left( {
\kappa}^{3}-3\,{\kappa}^{2}\lambda+2\,x_{{0}}{\kappa}^{2}-3\,\kappa\,
\lambda\,x_{{0}}+ x_0^{2}\kappa+2\,\kappa\,{\lambda}^{2}+{\lambda
}^{2}x_{{0}} \right) }{ \left( x_{{0}}+\kappa \right) ^{4}}}\epsilon-\left( x_{{0}}+\kappa \right) ^{-7}\\
&\quad &\times \bigg({\kappa}^{2} \left(\kappa - \lambda \right)  \left( 6\,{
\lambda}^{2}-6\,\kappa\,\lambda+{\kappa}^{2} \right) +\kappa\, \left( 
\kappa-\lambda \right)  \left( \kappa-2\lambda\right)  \left( -2
\,\lambda+3\,\kappa \right) x_{{0}}\\
&\quad &+ \left( -3\,{\lambda}^{3}-{\kappa}
^{2}\lambda+3\,{\kappa}^{3} \right)  x_0^{2}+\kappa\, \left( 
\kappa+3\,\lambda \right)  x_0^{3}
\bigg) {\epsilon}^{2}+\mathcal O(\epsilon^3),
\end{eqnarray*}
 then cf. \eqref{gamman} the span of the vector
\begin{eqnarray}
 v &=& \begin{pmatrix}
      \epsilon \phi^2\\
1+\epsilon \partial_x \eta^2 \phi^2
     \end{pmatrix}= \left( \begin {array}{c} 0\\ \noalign{\medskip}1\end {array} \right)+ 
 \left( \begin {array}{c} -{\frac { x_{{0}}+\kappa-\lambda
 }{x_{{0}}+\kappa}}\\ \noalign{\medskip}-{\frac {
 \left( x_{{0}}+\kappa-\lambda \right) \kappa}{ \left( x_{{0
}}+\kappa \right) ^{3}}}\end {array} \right)\epsilon \eqlab{vvector}\\
&\quad &+   \left( \begin {array}{c} {\frac {  \kappa\,
 \left(\kappa - \lambda \right)  \left( \kappa-2\lambda\right) +
 \left( 2\,\kappa-\lambda \right)  \left(\kappa - \lambda \right) x_{{0
}}+\kappa x_0^{2} }{ \left( x_{{0}}+\kappa \right) ^{4}}}
\\ \noalign{\medskip}{\frac { \left( \kappa\, \left( \kappa-\lambda \right)  \left( \kappa-3\,\lambda \right) + \left( 2\,{
\kappa}^{2}-\kappa\,\lambda-2\,{\lambda}^{2} \right) x_{{0}}+ \left( 
\kappa+3\,\lambda \right)  x_0^{2} \right) \kappa}{ \left( x_{{0}
}+\kappa \right) ^{6}}}\end {array}
\right){\epsilon}^{2}+\mathcal O(\epsilon^3)\nonumber
\end{eqnarray}
gives a third order approximation of the tangent space. We have left out the complicated $\mathcal O(\epsilon^3)$-terms. The transformation $(x,y)\mapsto (x_0,y_0)=(x+\epsilon \phi^2 y,y+\eta^2(x+\epsilon \phi^2 y))$ therefore transforms the system \eqref{eX0Y0mmh} into:
\begin{eqnarray}
 \dot x &=& \epsilon\left(\Lambda +\mu_3 y + \left(-{\frac { \left( x+\kappa-\lambda \right) \lambda}{
 \left( x+\kappa \right) ^{2}}}{\epsilon}+\mathcal O(\epsilon^2)
 \right) y^2+\mathcal O(\epsilon^2 y^3)\right),\eqlab{xysmmh}\\
\dot y&=&\rho(x) + \left(A(x)+\mathcal O(\epsilon^4)\right)y+\mathcal O(\epsilon y^2),\nonumber
\end{eqnarray}
 with $\rho$ as in \eqref{rmmh} and
\begin{eqnarray*}
 \mu_3 = -{\frac {x{\lambda}^{4} \left( {\kappa}^{2}-12\kappa x+15\,{x}^{2}
 \right)}{ \left( x+\kappa \right) ^{9}}}\epsilon^3
+\mathcal O(\epsilon^4),
\end{eqnarray*}
and where $\rho$, $\Lambda$ and $A$ are given in \eqref{rmmh}, \eqref{Lmmh} respectively \eqref{Ammh}.

{\subsection{Comparison with the results of the CSP method} \seclab{comp_mmh}
The reference \cite{kap4} applies the CSP method to \eqref{mmh}. Eqs. (5.56) and (5.57) in this document constitute the two components of a vector denoted by $A_1^{(2)}$ that is tangent to the fast fibers up to and including second order terms. This vector should therefore (up to a scaling) coincide with our $v$ \eqref{vvector}, omitting the $\mathcal O(\epsilon^3)$-remainder. The $\mathcal O(1)$-terms of $A_1^{(2)}$ and our $v$ coincide as $(0,1)$. Next, the $\mathcal O(\epsilon)$-term of $A_1^{(2)}$ is
\begin{eqnarray*}
 \begin{pmatrix}
  -\frac{s+\kappa-\lambda }{s+\kappa}\\
  -\frac{\kappa(s+\kappa-\lambda)}{(s+\kappa)^3}
 \end{pmatrix}
\end{eqnarray*}
while in \eqref{vvector} the corresponding term is
\begin{eqnarray*}
\left( \begin {array}{c} -{\frac { x_{{0}}+\kappa-\lambda
 }{x_{{0}}+\kappa}}\\ \noalign{\medskip}-{\frac {
 \left( x_{{0}}+\kappa-\lambda \right) \kappa}{ \left( x_{{0
}}+\kappa \right) ^{3}}}\end {array} \right)
\end{eqnarray*}
Realizing that their $s$ is our $x_0$ we see that these terms also coincide. Finally, for the $\mathcal O(\epsilon^2)$-terms we see that
\begin{eqnarray}
 \begin{pmatrix}
\frac{\kappa ( s+\kappa-2\lambda)(s+\kappa-\lambda)+\lambda^2 s}{(s+\kappa)^4}  \\
\frac{(s+\kappa-\lambda) (\kappa^2(s+\kappa-2\lambda)+\kappa\lambda s) +\kappa \lambda^2 s}{(s+\kappa)^6}
 \end{pmatrix}\eqlab{Aeps2}
\end{eqnarray}
in $A_1^{(2)}$ while in $v$ we have
\begin{eqnarray}
\left( \begin {array}{c} {\frac { \kappa\,
 \left(\kappa - \lambda \right)  \left( \kappa-2\lambda\right) +
 \left( 2\,\kappa-\lambda \right)  \left(\kappa - \lambda \right) x_{{0
}}+\kappa x_0^{2}  }{ \left( x_{{0}}+\kappa \right) ^{4}}}
\\ \noalign{\medskip}{\frac { \left( \kappa\, \left( \kappa-\lambda \right)  \left( \kappa-3\,\lambda \right) + \left( 2\,{
\kappa}^{2}-\kappa\,\lambda-2\,{\lambda}^{2} \right) x_{{0}}+ \left( 
\kappa+3\,\lambda \right)  x_0^{2} \right) \kappa}{ \left( x_{{0}
}+\kappa \right) ^{6}}}\end {array}
\right).\eqlab{veps2}
\end{eqnarray}
Clearly the denominators match, given that $s=x_0$. We therefore collect powers with respect to $s$ in the numerators in the components of \eqref{Aeps2}:
\begin{eqnarray*}
 \kappa ( s+\kappa-2\lambda)(s+\kappa-\lambda)+\lambda^2 s 
 &=&\kappa(\kappa-\lambda)(\kappa-2\lambda)+(2\kappa -\lambda)(\kappa-\lambda)s +\kappa s^2\\
 (s+\kappa-\lambda) (\kappa^2(s+\kappa-2\lambda)+\kappa\lambda s) +\kappa \lambda^2 s &=&\bigg( \kappa\, \left( \kappa-\lambda \right)  \left( \kappa-3\,\lambda \right) + \left( 2\,{
\kappa}^{2}-\kappa\,\lambda-2\,{\lambda}^{2} \right) s\\
&\quad &+ \left( \kappa+3\,\lambda \right) {s}^{2} \bigg) \kappa,
\end{eqnarray*}
and notice that these agree when $s=x_0$ with the numerators in \eqref{veps2}. 
}

\subsection{Numerical computations of $\eta$ and $\phi$}\seclab{numetaphimmh}
In \figref{mmh_y0} (a) and (b) we have compared the solution
$(x_0,y_0)=(x_b,y_b)$ of \eqref{eX0Y0mmh} initiated at the base point
$(x_{b0},\eta(x_{b0}))\in M$ with (i) the solution  $(x_l,y_l)$
initiated at $(x_{b0},\eta(x_{b0}))+v\vert v \vert^{-1}s$ (dashed) and
with (ii) (by the naive projection) the solution  $(x_n,y_n)$
initiated at $(x_{b0},\eta(x_{b0})+s)$ (full line) for different values
of $s$ and for a time integration of length $t = \epsilon^{-1}$. The subscripts $b$, $l$ and $n$ refer to \textit{base}, \textit{linear} projection respectively \textit{naive} projection. The
parameter $s$ measures the distance from the slow manifold and $v$ is the
tangent vector to the fiber at the base point $(x_{b0},\eta(x_{b0}))$
determined through $\phi$ and the equation \eqref{tsp}. We have set
$\kappa=2,\,\lambda=1$ and $x_{b0}=1.5$, and have computed ${\eta}$ and
$\phi$ numerically using the Eqs. \eqsref{fraser2}{phii} in \corref{cornum}. We
have used $4$ iterations on both $\eta$ and $\phi$ resulting in
error-fields of $\sim 10^{-7}$ respectively $\sim 10^{-10}$ for
$\epsilon=0.1$. The comparison is made through 
\begin{eqnarray}
\upsilon_l &=& \Vert (x_b,y_b)(\epsilon^{-1})-(x_l,y_l)(\epsilon^{-1})\Vert,\quad \text{(\textit{l}inear projection),}\nonumber\\
&\text{and}\nonumber\\
 \upsilon_n &=& \Vert (x_b,y_b)(\epsilon^{-1})-(x_n,y_n)(\epsilon^{-1})\Vert, \quad \text{(\textit{n}aive projection).}\nonumber
\end{eqnarray} 
In (a) $\epsilon=0.1$ while $\epsilon=0.01$ in (b). We see that $\upsilon_l\ll \upsilon_n$ and compute $\upsilon_l \approx \mathcal O(s^{2.009})$,  whereas $\upsilon_n \approx \mathcal O(s^{1.000})$.

\begin{figure}[h!]
\begin{center}
\subfigure[$\epsilon=0.1$]{\includegraphics[width=.475\textwidth]{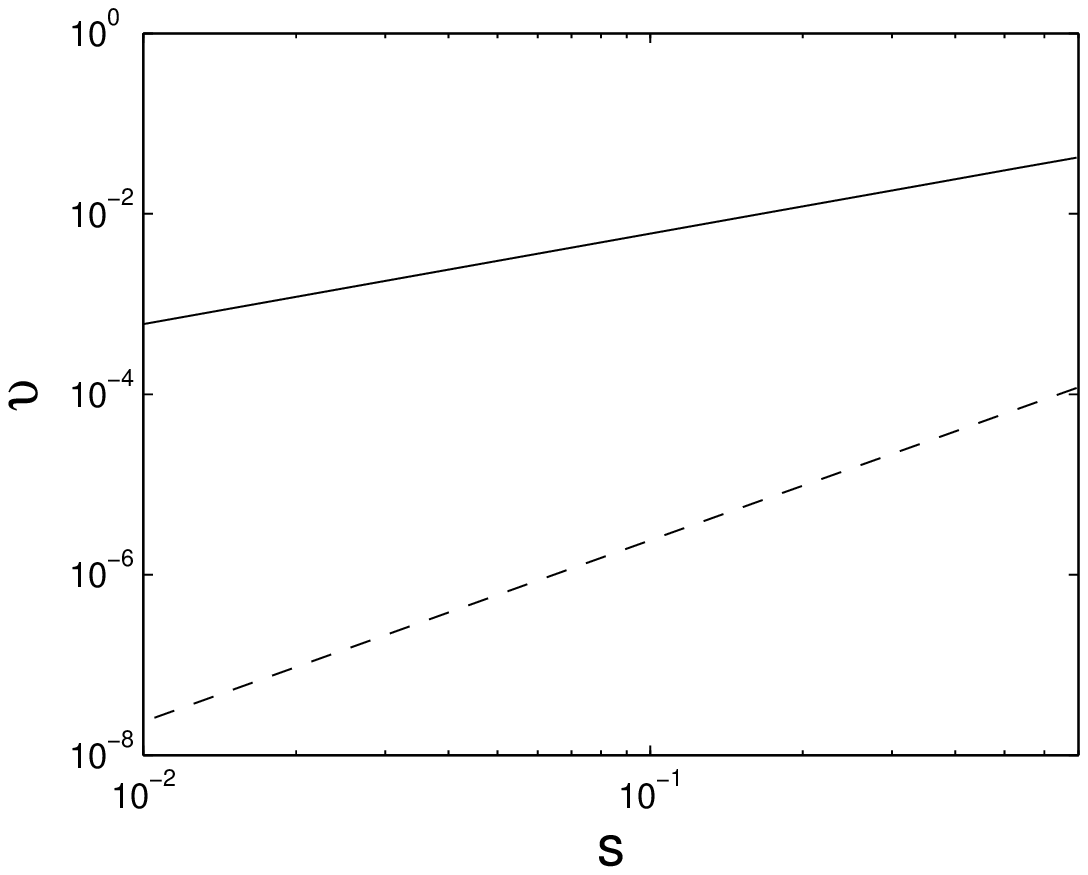}}
\subfigure[$\epsilon=0.01$]{\includegraphics[width=.475\textwidth]{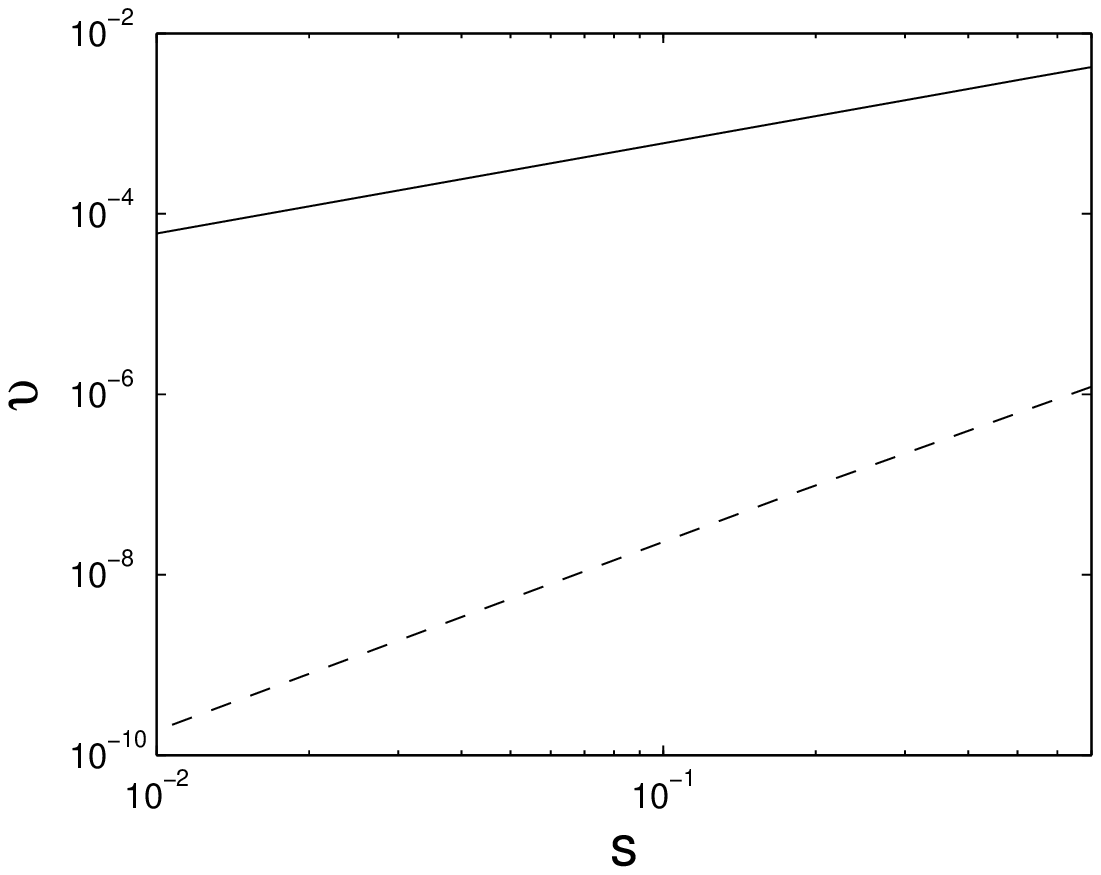}}
\end{center}
\caption{The errors from approximating fibers using the tangent spaces ($\upsilon_l$: dashed lines) and naive projections ($\upsilon_n$: full lines) as functions of the distance from the slow manifold. Here $\kappa=2$, $\lambda=1$ and the initial condition on $x_0=x_b$ is $x_{b0}=1.5$.}
\figlab{mmh_y0}
\end{figure}

The errors in $\eta^n$ and $\phi^n$:
\begin{eqnarray*}
 \mathcal E_\eta &=& \sup_{x}\,\vert -\epsilon \partial_x\eta X_0+Y_0\vert,\\ &\text{respectively}
\\ \mathcal E_\phi &=& \sup_{x}\,\vert \epsilon (\partial_x X_0+\partial_y X_0 \partial_{x} {\eta})\phi -\epsilon \partial_x \phi X_0+\partial_y X_0 - \phi (-\epsilon \partial_{x} {\eta} \partial_y X_0+\partial_y Y_0)\vert,
\end{eqnarray*}
 for $\epsilon=0.1$ are shown in \figref{errfld_eps01} as a function of the iteration number. These are relevant errors since if $\mathcal E_\eta=0$ then $y_0=\eta$ is an exact slow manifold (cf. \eqref{fraser2}) and if $\mathcal E_\phi=0$ then the transformation $x_0=x_1+\epsilon\phi(x_1) y$ removes the term in $\dot x_1$ that is linear in $y$ exactly (cf. \eqref{phii}). The error in $\phi$ and $\eta$ are observed to be the order of machine precision $\sim 10^{-14}$ after $8$ respectively $10$ iterations. There is no or little improvement beyond this number. 
 It should also be mentioned that to approximate derivatives we use the five-point stencil:
\begin{eqnarray*}
 f'(x) \approx \frac{1}{2h}(-f(x+2h)+8f(x+h)-8f(x-h)+f(x-2h)),
\end{eqnarray*}
 the error being $\frac{h^4}{30}f^{(5)}(x_0)=\mathcal O(h^4)$, $x_0\in [x-2h,x+2h]$. We have used $h\approx 10^{-2}$ which gives an error of $\sim 10^{-8}$.

\begin{figure}[h!]
\begin{center}
\includegraphics[width=.5\textwidth]{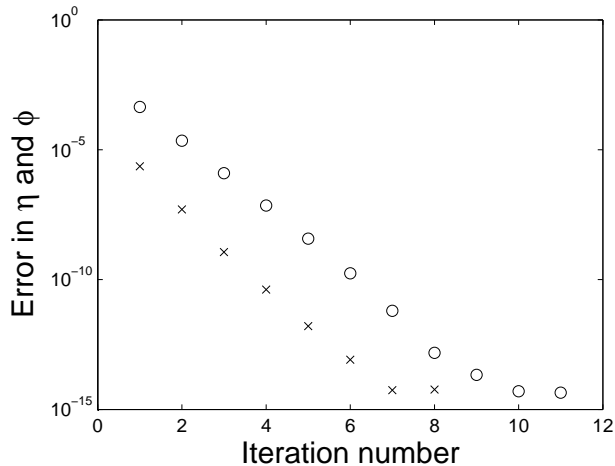}
\end{center}
\caption{The errors $\mathcal E_\eta$ $(\circ)$ and $\mathcal E_\phi$ ($\times$) for $\epsilon=0.1$ as a function of the iteration number. The functions $\eta$ and $\phi$ are computed numerically using the Eqs. \eqsref{fraser2}{phii} in \corref{cornum}. Here $\kappa=2$, $\lambda=1$ and the initial condition on $x_0$ is $x_{b0}=1.5$. In the computations leading to \figref{mmh_y0} we have used $4$ iterations on both $\eta$ and $\phi$. }
\figlab{errfld_eps01}
\end{figure}

In \figref{mmh_eps} we have taken $s=0.5$ and consider $\upsilon_l$ and $\upsilon_n$ as functions of $\epsilon$. Again, we see that $\upsilon_l\ll \upsilon_n$ and compute $\upsilon_l \approx \mathcal O(\epsilon^2)$ whereas $\upsilon_n \approx \mathcal O(\epsilon)$. This discrepancy is, however, exceptional as it is due to the fact that the Michaelis-Menten-Henri system is linear in the fast variable $y_0$. 

\begin{figure}[h!]
\begin{center}
{\includegraphics[width=.5\textwidth]{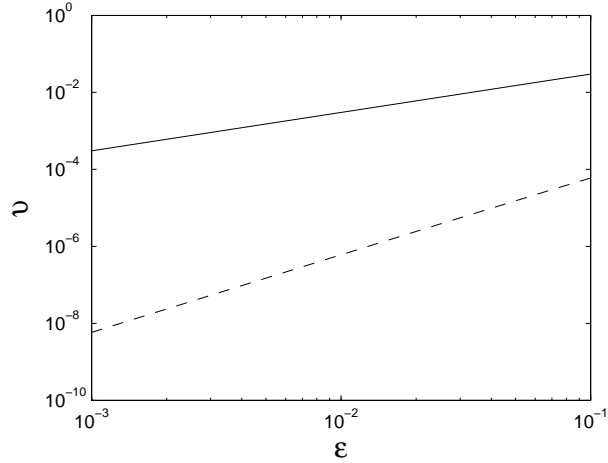}}
\end{center}
\caption{The errors from approximating fiber directions using the tangent spaces (dashed lines) and naive projections (full lines) as functions $\epsilon$. Here $s=0.5$, $\kappa=2$, $\lambda=1$ and the initial condition on $x_0$ is $x_{b0}=1.5$. We have used $4$ iterations in the computations of $\eta$ and $\phi$.}
\figlab{mmh_eps}
\end{figure}

Finally, solutions $(x_b,y_b)$ and $(x_l,y_l)$ of
\eqref{eX0Y0mmh} are shown in \figref{fiber_exeps0_4}. The solution
$(x_l,y_l)$ is initiated on the fiber of the base point with
$x_{b0}=1.5$, at a distance $s\approx 0.52$ from the base point. The full
line near $y_0=0$ is the slow manifold. The solutions $(x_b,y_b)$ and
$(x_l,y_l)$ are at $9$ different times indicated by $\times$
respectively $\circ$'s. For illustrative purposes we have chosen the
relatively large value of $\epsilon=0.4$. It is observed that, at least
approximately, the solution $(x_l,y_l)$ contracts along the fiber
directions indicated by the dashed lines moving from upper left to lower
right. The fiber directions are approximated as hyperplanes through $\phi$.

\begin{figure}[h!]
\begin{center}
{\includegraphics[width=.5\textwidth]{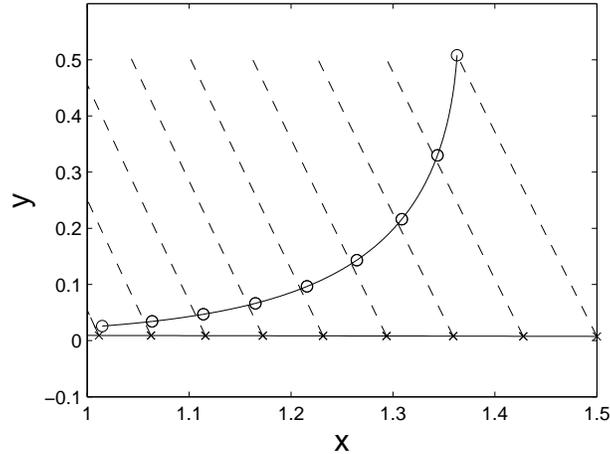}}
\end{center}
\caption{Solutions $(x_b,y_b)$ and $(x_l,y_l)$ of \eqref{eX0Y0mmh} for $\epsilon=0.4$.  The solution $(x_l,y_l)$ is initiated on the fiber corresponding to the base point with $x_{b0}=1.5$, at a distance $s\approx 0.52$ from the base, and is observed to contract to the solution $(x_b,y_b)$ along the fiber directions. The fiber directions are indicated by the dashed lines running from upper left to lower right.}
\figlab{fiber_exeps0_4}
\end{figure}

\section{Approximating the curvature of the fibers}\seclab{curve}
In this section we show how the SOF method may be further extended to also approximate the curvature of the fibers. According to \thmref{thm1} the transformation
\begin{eqnarray*}
 y\mapsto y_0 = y+\eta(x_0),\quad x^0\mapsto x_0 = x^0+\epsilon \phi(x^0) y,
\end{eqnarray*}
generated by the SOF method, transforms \eqref{xys12} into \eqref{xys13}:
\begin{eqnarray}
 \dot x^0 &=& \epsilon (\Lambda(x^0)+Q_0(x^0)y^2
+C(x^0,y)),\eqlab{Q0}\\
\dot y&=&A(x^0)y+R(x^0,y),\nonumber
\end{eqnarray}
with $C=\mathcal O(y^3)$ and $R=\mathcal O(y^2)$ up to exponentially small error. Note how we, as promised in the introduction, use superscripts $0$ to indicate the beginning of a new iteration. One can obtain an explicit expression for the quadratic term
$Q_0y^2$, a vector of symmetric bilinear forms, in terms of the known
functions: $X_0$, $Y_0$, $\eta$ and $\phi$, through the equation $x_0 =
x^0+\epsilon \phi(x^0) y$. 
 We will write the $i$th component of $Q_0y^2$ as
\begin{eqnarray}
 (Q_0y^2)_i = \langle y,Q_0^i y\rangle,\quad 1\le i\le n_s,\eqlab{Q0i}
\end{eqnarray}
where $Q_0^i=Q_0^i(x^0)$ is a symmetric $n_f\times n_f$-matrix. Recall that
we use the notation $(z)_i$ to denote the $i$th component of a vector $z$.
Here we have also introduced the real inner product $$\langle a,b\rangle =
\sum_{i=1}^{n_f} (a)_i (b)_i.$$ By introducing 
\begin{eqnarray*}
 x^0 = x^1 + \epsilon \psi_0(x^1)y^2,
\end{eqnarray*}
with $\psi_0$ a vector of symmetric bilinear forms, we therefore obtain
\begin{eqnarray}
 \dot x^1 = (I_s-\epsilon\partial_{x}\psi_0 y^2+J^{-1}(\epsilon \partial_{x}\psi_0 y^2)^2)\epsilon (\Lambda + \left\{\epsilon \partial_x \Lambda\psi_0 y^2+Q_0y^2 -2\psi_0(y)(A y)\right\}+\mathcal O(y^3)),\eqlab{xpc}
\end{eqnarray}
where $J=I_s+\epsilon \partial_{x}\psi_0 y^2$ is the Jacobian of the transformation $x^1\mapsto x^0$. Here $\psi_0(y)(A y)$ is understood as $$(\psi_0(y)(A y))_i =\frac12 \langle y,A^T \psi_0^i y\rangle + \frac12 \langle y,\psi_0^i A y\rangle,$$ using the notation in \eqref{Q0i}. The new error, that is the term in $\epsilon^{-1}\dot x^1$ which is quadratic in $y$, can again be decomposed into two separate contributions. One term comes from the expansion of $\dot x^0 - (\partial_y (\psi_0 y^2)) \dot y$, the curly bracket in \eqref{xpc}, while the other one is due to the inverse of the Jacobian. As for the linear case, we choose the unknown function $\psi_0$ so that the curly bracket in \eqref{xpc} vanishes for all $y$. 
This gives
\begin{eqnarray}
 \epsilon \sum_{j=1}^{n_s} \partial_{(x)_j} (\Lambda)_i \psi_0^j +Q_0^i -A^T \psi_0^i-\psi_0^i A = 0.\eqlab{psi0}
\end{eqnarray}
By Theorem 4.4.6 in \cite{hor1} this system has a unique solution
$\psi_0^i$ for $\epsilon=0$ iff $\sigma(A)\cap \sigma(-A)=\emptyset$.
Therefore we must exclude the elliptic case and \textit{the neutral saddle scenario}
where both $\lambda$ and $-\lambda$, $\text{Re}\,\lambda\ne 0$, are
eigenvalues of $A$. 
Note moreover that by taking transposes:
\begin{eqnarray*}
Q_0^i -A^T (\psi_0^i)^T-(\psi_0^i)^T A = 0,
\end{eqnarray*}
verifying that the solution is symmetric. The solution perturbs to a symmetric solution for $\epsilon\ne 0$ but small; $\epsilon \sum_{j=1}^{n_s} \partial_{(x)_j} (\Lambda)_i \psi_0^j$ is also symmetric. The solution satisfies
\begin{eqnarray*}
 \Vert \psi_0 \Vert_\chi 
\le K
\Vert Q_0\Vert_\chi,
\end{eqnarray*}
for some constant $K$ depending on $A^{-1}$. Then 
the new error becomes 
\begin{eqnarray*}
  Q_1 &=& -\epsilon\partial_{x}\psi_0 \Lambda,\quad 
\Vert Q_1 \Vert_{\chi-\xi} \le \frac{\epsilon K C_\Lambda}{\xi} \Vert Q_0\Vert_{\chi},
\end{eqnarray*}
which also vanishes at exact equilibria where $\Lambda \equiv 0$. As for the linear case we have that $\Lambda_1=\Lambda$ and $A_1=A$. Using such transformations successively it is therefore possible to approximate the curvature of the fibers up to exponentially small error. Formally we proceed as in the proof of \thmref{thm1}; the most important ingredient in the proof being the continued reduction of the domain together with the applications of Cauchy estimates to control the derivatives. 
\begin{theorem}\thmlab{thm2}
Assume that the assumptions of \thmref{thm1} hold true and that we have used \textnormal{the SOF} method to transform \eqref{xys12} into \eqref{xys13}. Assume furthermore that $\sigma(A)\cap \sigma(-A)=\emptyset$. Fix $0\le \underline{\underline{\chi}}<\underline{\chi}$. Then there exist an $\epsilon^0\le \epsilon_0$, where $\epsilon_0$ is from \thmref{thm1}, and an $N_3=\mathcal O(\epsilon^{-1})\in \mathbb N$ so that for all $\epsilon\le  \epsilon^0$ the sequence of transformations 
 $x^n= x^{n+1} + \epsilon \psi_n(x^{n+1})y^2$, $0\le n\le N_3-1$, where $\psi_n^i\in \mathbb R^{n_f\times n_f}$ solves
\begin{eqnarray}
\epsilon \sum_{j=1}^{n_s} \partial_{(x)_j} (\Lambda)_i \psi_n^j +Q_n^i -A^T \psi_n^i-\psi_n^i A = 0,\quad 1\le i\le n_s,\eqlab{psieqn}
\end{eqnarray}
the quantity $$\epsilon Q_ny^2=\left\{\begin{array}{cc}
                     \textnormal{given by Eq}.\, \eqref{Q0}\quad &\text{for
$n=0$},\\
-\epsilon^2 \partial_x \psi_{n-1} \Lambda y^2 &\text{for $n\ge 1$},
                    \end{array}\right.
$$ 
being the term in the expression for $\dot x^n$ which is quadratic in $y$, eventually transforms \eqref{xys13} into
\begin{eqnarray}
 \dot x^{N_3}& =&\epsilon (\Lambda(x^{N_3})+\tilde C(x^{N_3},y))+\mathcal O(e^{-\tilde c_1/\epsilon}),\eqlab{xys14}\\ 
\dot y &=&A(x^{N_3})y+\tilde R(x^{N_3},y)  +\mathcal O(e^{-\tilde c_1/\epsilon}).\nonumber
\end{eqnarray}
Here $(x^{N_3},y)\in (\mathcal U+i\underline{\underline{\chi}})\times (\mathcal V+i{\underline{\nu}})$, $\tilde C=\mathcal O(y^3)$ and
 \begin{eqnarray*}
 \Vert \tilde C-C\Vert_{\underline{\underline{\chi}},\underline{\nu}},\,\Vert \tilde R-R\Vert_{\underline{\underline{\chi}},\underline{\nu}}&\le \tilde c_2 \epsilon,
\end{eqnarray*}
for some constants $\tilde c_1$ and $\tilde c_2$. Also the $\mathcal O(e^{-\tilde c_1/\epsilon})$ error terms in \eqref{xys14} vanish at true equilibria. 

The transformation $x \mapsto x^0=x+\epsilon
\psi(x) y^2$ with 
\begin{eqnarray}
\psi = \sum_{i=0}^{N_3-1}
\psi_i,\eqlab{psi}
\end{eqnarray}
differs from the {composition} of $x^{N_3}\mapsto \cdots \mapsto x^1\mapsto x^{0}$ by
$\mathcal O(y^3)$-terms and the equations for $(x,y)$
therefore takes a similar form to \eqref{xys14}: The set $\{y=0\}$
is \textnormal{almost} invariant and the $y$-space provides an
\textnormal{almost} $y^2$-approximation to the fibers.
In terms of the $(x_0,y_0)$-variables this quadratic approximation,
parametrized by $y$, takes the following form:
\begin{eqnarray}
  x_0 &=& x+\epsilon \phi y+\epsilon \psi y^2,\eqlab{curveapp}\\
y_0&=&y+\eta+\partial_x \eta(\epsilon \phi y+\epsilon \psi y^2)+\frac12 \partial_x^2 \eta (\epsilon \phi y)^2,\nonumber
\end{eqnarray}
with all functions on the right hand sides evaluated at $x$, for the base point $(x,\eta(x))$. 
\end{theorem}
{\begin{remark}
We highlight that this theorem also holds true for the saddle type slow manifolds where $A$ \eqref{A} has eigenvalues with both negative and positive real parts. We just have to exclude the {neutral saddle} scenario. This is perhaps rather surprising seeing that the Fenichel normal form \eqref{fnf2} takes a slightly different form: $\dot u$ includes quadratic terms of the form $vw$. However, \eqref{curveapp} provides absolutely no control of the location of the stable and unstable manifolds. 

For the neutral saddle case, one may
 follow the general philosophy of normal form theory and relax the requirements of the transformations and accordance with \eqref{fnf2} seek only to remove the terms in the slow vector field that are quadratic in the fast variables (say $y_s^2$, playing the role of $v$ in \eqref{fnf2}) associated with the contraction respectively the fast variables associated with the expansion (say $y_u^2$, playing the role of $w$ in \eqref{fnf2}) from the slow manifold. That is one would leave quadratic terms of the form $y_sy_u$ behind. This procedure requires a change of basis to split the fast variables into $y_s$ and $y_u$ and as such it does not fit within the procedures we have developed in this paper. We therefore leave out the details.
 
 Finally, we point out that \eqref{psieqn} for the determination of $\psi$ \eqref{psi} by summation over $n$ can be written in form similar to Eqs. \eqsref{fraser2}{phii}. In contrast to \eqref{phii} we here need the first and second partial derivatives of the vector-field. 
%
%
\end{remark}}

\subsection{Analytic expression of $\psi$ to $2$nd order for the Michaelis-Menten-Henri model}
We now apply this principle to the Michaelis-Menten-Henri model. We start from \eqref{xysmmh} where the quadratic term in $\epsilon^{-1}\dot x$ is of order $\mathcal O(\epsilon)$:
\begin{eqnarray*}
 Q_1 = -{\frac { \left( x+\kappa-\lambda \right) \lambda\,}{ \left( x
+\kappa \right) ^{2}}}\epsilon+\mathcal O(\epsilon^2).
\end{eqnarray*}
We have therefore denoted this term by $Q_1$ rather than $Q_0$. Then $\psi_1$ solves \eqref{psieqn} with $n=1$ and $\psi_0=0$:
\begin{eqnarray*}
 \psi_1 &=& \frac{Q_1}{2A-\epsilon \partial_{x} \Lambda}={\frac { \left( x+\kappa-\lambda \right) \lambda}
{ 2\left( x+\kappa \right) ^{3}}}{\epsilon}+\mathcal O(\epsilon^2),
\end{eqnarray*}
so that
\begin{eqnarray*}
Q_2 &=&-\partial_x \psi_1 \Lambda = -{\frac {{\lambda}^{2} \left( 2(x+\kappa)-3\,\lambda \right) x
}{ 2\left( x+\kappa \right) ^{5}}}{\epsilon}^{2}
+\mathcal O(\epsilon^3).
\end{eqnarray*}
Finally 
\begin{eqnarray*}
 \psi_2 &=&\frac{Q_2}{2A-\epsilon \partial_{x} \Lambda}=\,{\frac {{\lambda}^{2} \left( 2(x+\kappa)-3\,\lambda \right)
x}{ 4\left( x+\kappa \right) ^{6}}}{
\epsilon}^{2}
+\mathcal O(\epsilon^3).
\end{eqnarray*}
Let $\psi^2 = \psi_1+\psi_2$: 
\begin{eqnarray*}
\psi^2 &=&{\frac { \left( x+\kappa-\lambda \right) \lambda}{2
 \left(x+\kappa \right) ^{3}}}\epsilon+\frac14 \left(x+\kappa
 \right) ^{-6} \bigg(2\,\kappa \left( -\lambda+\kappa \right)  \left( 8\,{
\lambda}^{2}-9\,\kappa\,\lambda+2\,{\kappa}^{2} \right) \\
&\quad &+ \left( 38\,\kappa\,{\lambda}^{2}-44\,{\kappa}^{2}\lambda
+12\,{\kappa}^{3}-7\,{\lambda}^{3} \right)x +\left( 4\,{
\lambda}^{2}+12\,{\kappa}^{2}-22\,\kappa\,\lambda \right) {x}^{2}+4\kappa{
x}^{3}\bigg)
+\mathcal O(\epsilon^3),
\end{eqnarray*}
then, in terms of the original $(x_0,y_0)$-variables in \eqref{eX0Y0mmh}, we have cf. \eqref{curveapp} obtained the following quadratic approximation, correct up to terms including $\epsilon^3$, of the fiber with base point $(x,\eta(x))$:
\begin{eqnarray*}
 x_0 &=& x+\epsilon \phi^2(x)y+\epsilon \psi^2( 
x) y^2\\
&=&x+ \left( -{\frac { \left( x+\kappa-\lambda \right) }{x+\kappa
}}\epsilon+{\frac {\kappa\, \left( -\lambda+\kappa \right)  \left( \kappa-2\,
\lambda \right) + \left( 2\,\kappa-\lambda \right)  \left( -\lambda+
\kappa \right)x+\kappa {x}^{2}}{ \left(x+\kappa \right) ^{4}}}
{\epsilon}^{2}+\mathcal O(\epsilon^3)
 \right) y\\
&\quad &+ \left( {\frac { \left(x+\kappa-\lambda \right) \lambda}{ 2\left(x+\kappa \right) ^{3}}}
{\epsilon}^{2}+\mathcal O(\epsilon^3) \right) {y}^{2},\\
y_0 &=&y+\eta^2+\partial_x \eta^2(\epsilon \phi^2y+\epsilon \psi^2 y^2)+\frac12 \partial_x^2 \eta^2 (\epsilon \phi^2y)^2\\
&=&{\frac {\kappa\,x\lambda}{ \left(x+\kappa \right) ^{4}}}\epsilon-{
\frac {\kappa\,\lambda\,x \left( {\kappa}^{2}+x\kappa-2\,\kappa\,
\lambda+3\,\lambda\,x \right) }{ \left(x+\kappa
 \right) ^{7}}}{\epsilon}^{2}+\mathcal O(\epsilon^3)+ \bigg( 1-{\frac { \left(x+\kappa-\lambda \right) 
 \left( \kappa-3x \right) \kappa\,\lambda}{ \left(x
+\kappa \right) ^{6}}}{\epsilon}^{2} \\&\quad &+\mathcal O(\epsilon^3) \bigg) y+\left(2\,{\frac { \left(x+\kappa-\lambda
 \right)  \left( 2\,\kappa-3 x \right) \kappa\,\lambda\,}{ \left(x+\kappa \right) ^{7}}}{\epsilon}^{2
}+\mathcal O(\epsilon^3)\right){y}^{2},
\end{eqnarray*}
parametrized by $y$. For simplicity, we have here chosen to omit $\mathcal
O(\epsilon^3)$-terms. In particular, $(x,y)=(\tilde x+\epsilon\psi^2(\tilde
x)\tilde y^2,\tilde y)$ transforms \eqref{xysmmh} into:
\begin{eqnarray*}
 \dot{\tilde x} = \epsilon\left(\Lambda(\tilde x) +\mu_3(\tilde x) \tilde y + Q_3 \tilde y^2+\mathcal O(\epsilon^2\tilde y^3)
 \right),\\
\dot{\tilde y}=\rho(\tilde x) + \left(A(\tilde x)+\mathcal O(\epsilon^4)\right)\tilde y+\mathcal O(\epsilon \tilde y^2),
\end{eqnarray*}
with 
\begin{eqnarray*}
 Q_3(\tilde x) ={\frac { \left( {\tilde x}+\kappa-\lambda \right) \kappa\,{\lambda
}^{2}}{ 2\left( {\tilde x}+\kappa \right) ^{5}}}{\epsilon}^{3}
+\mathcal O(\epsilon^4).
\end{eqnarray*}

\subsection{Numerical computation of $\psi$}
\figref{mmh_y03_eps01} shows the error
\begin{eqnarray*}
 \upsilon = \Vert (x_b,y_b)(\epsilon^{-1})-(x_q,y_q)(\epsilon^{-1})\Vert
\end{eqnarray*}
with $(x_q,y_q)$ being the solution initiated at points along the \textit{quadratic approximation} of the fibers with base $(x_b,y_b)$ obtained from numerically computing the $\psi_i$'s, 
as a function of the distance $s\in [0.5, 10]$ from the slow manifold. We have used $4$ iterations in computing $\psi$ giving rise to an error of $\sim 10^{-8}$ for $\epsilon=0.1$. The error decreases as $\approx \mathcal O(s^{3.033})$ in agreement with the analysis. 

\begin{figure}[h!]
\begin{center}
{\includegraphics[width=.5\textwidth]{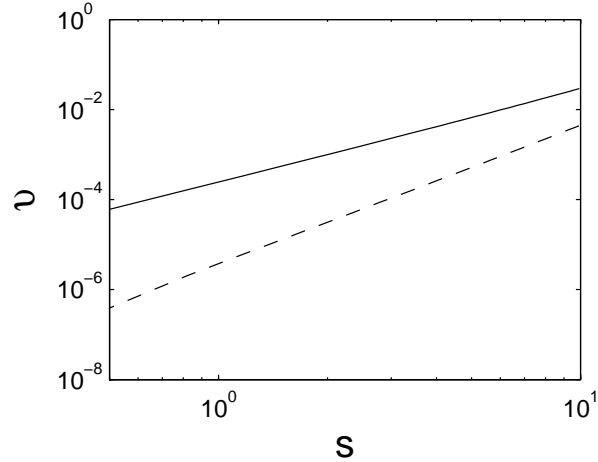}}
\end{center}
\caption{The error from approximating the fibers in the Michaelis-Menten-Henri model using a linear (full line) and quadratic (dashed line) approximation of the fiber as a function of the distance $s\in [0.5, 10]$ from the slow manifold. Here $\kappa=2$, $\lambda=1$, $\epsilon=0.1$ and the initial condition on $x$ is $x_b^0=1.5$. The error from the quadratic approximation decreases as $\approx \mathcal O(s^{3.033})$. We have used $4$ iterations in the computations of $\eta$, $\phi$ and $\psi$. }
\figlab{mmh_y03_eps01}
\end{figure}

\section{Lindemann mechanism} \seclab{linde}
In this section we consider the Lindemann mechanism 
\begin{eqnarray}
 \dot x &=&X(x,y)=-x(x-y),\eqlab{lm}\\
 \dot y &=&Y(x,y)=x(x-y)-\epsilon y,\nonumber
\end{eqnarray}
and assume as usual $0< \epsilon\ll 1$. We will consider $x\ge 0$ and $y\ge 0$ and note here that $x=0=y$ is the unique equilibrium. Note also how the norm of $\dot x$ is not slow throughout phase space. We have therefore denoted it by $X$ rather than $\epsilon X$. This form does therefore not directly apply to our setting. In \cite{gou1} the authors apply the CSP method and a modified ``CSP-like'' method, which is based on the SO method, to this problem and show that they can lead to a simplified non-stiff system. We will aim at something similar here, highlighting that even though the equations are not in the form of \eqref{xys0} the SOF method can still be applied. We will as for the Michaelis-Menten-Henri mechanism also compare our results with the CSP method.

The crucial thing for the success of the method is the existence of a transformation:
\begin{eqnarray}
 \begin{pmatrix}
  w\\
  z 
 \end{pmatrix} = \begin{pmatrix}
  x+y\\
  2y 
 \end{pmatrix}\eqlab{wz}
\end{eqnarray}
transforming the non-standard equations into
\begin{eqnarray}
 \dot w &=&\epsilon W(w,z)=- \frac12 \epsilon z,\eqlab{wzeqn}\\
 \dot z &=& Z(w,z)=2w^2-(3w+\epsilon)z+z^2,\nonumber
\end{eqnarray}
taking the form \eqref{xys0}. Here $Z(w,z)=0$ gives a slow manifold 
\begin{eqnarray*}
z = w+\mathcal O(\epsilon).
\end{eqnarray*}
with $\mathcal O(\epsilon)$ error. {The graph $z=w$ corresponds to $y=x$ which is a sub-space filled with equilibria for $\epsilon=0$. This is directly related with the existence of the transformation in \eqref{wz}.} 

The variable $z$ is truly fast near this graph provided {$w\ge c> 0$, $c$ independent of $\epsilon$}, so that $\partial_z Z\ne 0$. This graph can by \eqref{wz} be parametrized by $x$ as $y=x+\mathcal O(\epsilon)$. This is enough for the SOF method to apply to \eqref{lm}. 

\subsection{Approximating the slow manifold}\seclab{lindeappsm}
To start the SOF method, consider the equation $Y(x,y)=0$ having the solution $y=\eta_0(x)=\frac{x^2}{x+\epsilon}=x+\mathcal O(\epsilon)$. Then through $(x,y)=(x_0,\eta_0(x_0)+y_0)$ we obtain
\begin{eqnarray*}
 \dot x_0 &=&X_0(x_0,y_0)=x_0y_{{0}}-{\frac {\epsilon\,x_0^{2}}{x_0+\epsilon}}
,\\
 \dot y_0 &=&Y_0(x_0,y_0)={\frac {x_0^{3} \left( x_0+2\,\epsilon \right) \epsilon}{
 \left( x_0+\epsilon \right) ^{3}}}-{\frac { \left( 2\,x_0^{
3}+5\,\epsilon\,x_0^{2}+3\,{\epsilon}^{2}x_0+{\epsilon}^{3}
 \right) y_{{0}}}{ \left( x_0+\epsilon \right) ^{2}}}.
\end{eqnarray*}
Note that $\rho_0(x_0)={{x_0^{3} \left( x_0+2\,\epsilon \right)}{
 \left( x_0+\epsilon \right) ^{-3}}} \epsilon$ is small and so $y_0=0$ is close to being invariant. We also point out that neither $y_0$ nor $x_0$ are fast near $x_0=0=y_0$ as the Jacobian of the vector-field at this point is
 \begin{eqnarray*}
  \partial_{(x,y)} \begin{pmatrix}
                    X_0\\
                    Y_0
                   \end{pmatrix}(0,0) = \begin{pmatrix}
                   0 & 0\\
0 &-\epsilon\end{pmatrix}.
 \end{eqnarray*}
 Next, we set $Y_0(x_0,y_0)=0$ and obtain
\begin{eqnarray}
 y_0=\eta_1(x_0)={\frac {x_0^{3} \left( x_0+2\epsilon \right) \epsilon}{
 \left( x_0+\epsilon \right)  \left( 2x_0^{3}+5\epsilon
x_0^{2}+3{\epsilon}^{2}x_0+{\epsilon}^{3} \right) }}=\frac{\epsilon}{2}-\frac{3\epsilon^2}{4x}+\mathcal O(\epsilon^3),\eqlab{eta1lm}
\end{eqnarray}
as our next approximation of the slow manifold. 
This gives new equations of the form
\begin{eqnarray}
\dot x_0 &=&X_1(x_0,y_1) = {\frac {\epsilon x_0^{2} \left( x_0+\epsilon \right) ^{2}}
{2 x_0^{3}+5 \epsilon x_0^{2}+3 {\epsilon}^{2}x_0+
{\epsilon}^{3}}}+x_0y_1,\nonumber\\
 \dot y_1 &=&Y_1(x_0,y_1)=\rho_1(x_0)+A_1(x_0)y_1,\nonumber
\end{eqnarray}
with 
\begin{eqnarray}
 \rho_1(x_0) ={\frac { \left( 3 x_0^{4}+16 x_0^{3}\epsilon+28 {
\epsilon}^{2}x_0^{2}+20 {\epsilon}^{3}x_0+6 {\epsilon}^{4}
 \right) {\epsilon}^{3}x_0^{4}}{ \left( 2 x_0^{3}+5 
\epsilon x_0^{2}+3 {\epsilon}^{2}x_0+{\epsilon}^{3}
 \right) ^{3}}},\eqlab{rho1lm}
\end{eqnarray}
and $y_1$ given as, $y_0=\eta_1(x_0)+y_1$, the deviation from $y_0=\eta_1(x_0)$. Note how $\rho_1$ 
(cf. \eqref{rho11}) is the product of $-\partial_x \eta_1$, which is $\mathcal O(\epsilon)$ based on \eqref{eta1lm}, and $X_0(x_0,\eta_1(x_0))=-{ {x_0^{2}\epsilon}{(x_0+\epsilon)^{-1}}}$ and therefore it is order $\mathcal O(\epsilon^2)$. However, it is even $\mathcal O(\epsilon^3)$, and this is due to $\eta_1(x_0)=\frac{\epsilon}{2}-\frac{3\epsilon^2}{4x_0}+\mathcal O(\epsilon^3)$ such that $\partial_x \eta_1 =\mathcal O(\epsilon^2)$. Therefore \eqref{eta1lm} gives a slow manifold accurate up to terms including $\epsilon^2$. Surprisingly, even though $y_1$ is not fast near $x_0=0$, the graph $y_1=0$ is still close to being invariant there as $\rho_1(0)=0$. The slow manifold obtained through the SO method always includes nearby equilibria. In particular, there is an improvement in the error as we approach $x_0=0$. The solutions shown in \cite{gou1} do not have this property as the truncation of the expansion about $\epsilon=0$:
\begin{eqnarray*}
 \rho_1(x_0)=\frac{3\epsilon^3}{x_0}+\mathcal O(\epsilon^4),
 \end{eqnarray*}
 does not preserve $\rho_1(0)=0$. 

 The approximation 
 \begin{align}
 y=\eta^1(x_0) = \eta_0(x_0)+\eta_1(x_0) =  {\frac {x_0^{2} \left( {\epsilon}^{2}+4x_0\epsilon+2x_0^{2} \right) 
}{{\epsilon}^{3}+3{\epsilon}^{2}x_0+5x_0^{2}\epsilon+2x_0^{3}}},\eqlab{eta11lm}
 \end{align}
 defines a slow manifold since the vector-field of the reduced system:
\begin{eqnarray*}
 \epsilon^{-1} \dot x_0 = \Lambda_1(x_0)=-{\frac {x_0^{2} \left( x_0+\epsilon \right) ^{2}}
{2 x_0^{3}+5 \epsilon x_0^{2}+3 {\epsilon}^{2}x_0+
{\epsilon}^{3}}}
=-\frac{1}{2}x_0+\frac{1}{4}\epsilon+\mathcal O(\epsilon^2).
\end{eqnarray*}
has small norm. This equation is identical to Eq. (80a) in \cite{gou1} (their $z$ is our $x_0$ and their time is $\epsilon$ times ours). 
 \subsubsection{Comparison with the CSP method}
 To compute the CSP-approximation of the slow manifold we follow \cite{gou1} but use the principle from \cite{kap4} (see Eq. 3.29) of inserting the previous guess into the $B_1^{(q)}$ vector when computing the CSP condition to avoid the cubic equation in Eq. (100) in \cite{gou1}. The second order CSP-approximation obtained is identical to the SO-approximation in \eqref{eta1lm}. To compare the two methods we therefore consider the errors (measured by the resulting $\rho_2$'s) from the third order approximations. This is illustrated in \figref{linde} for $\epsilon=0.1$ (SO: full line, CSP: dotted line). The expressions are lengthy so we prefer not to include them. It is clear that both approximations improve as $x_0$ approaches $0$ at the same rate $\mathcal O(x_0^5)$. This property has to the authors knowledge not been verified in general for the CSP method. In this case, the error from the SO-approximation is smaller for $x_0<0.11$ while the CSP-approximation performs better for larger values of $x_0>0.11$. 
 We used high precision (50 digits) calculations in Maple. 
%
%
 
 
%
\begin{figure}[h!]
\begin{center}
{\includegraphics[width=.5\textwidth,trim= 0 130 0 130]{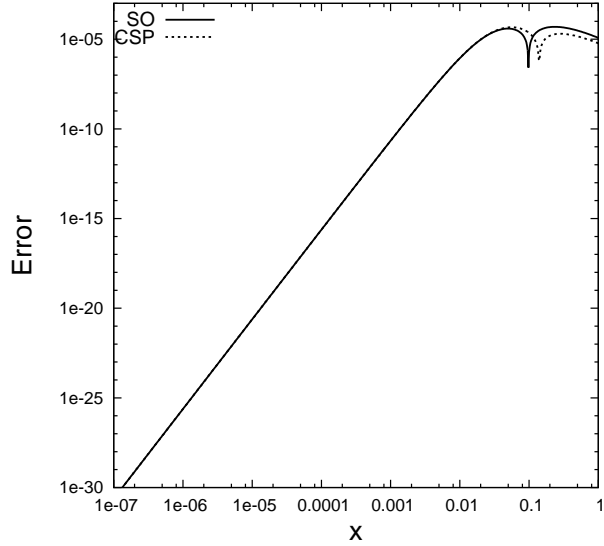}}
\end{center}
\caption{The error from approximating the slow manifold using the CSP method (dotted line) and the SO method (full line) for $\epsilon=0.1$. Both methods have been applied three times. The figure illustrates that the two approximations both improve near the equilibrium $(x,y)=(0,0)$. This is somewhat surprising since the motion normal to the slow manifold approximations are not fast near this point. The premises of both method therefore breaks down.}
\figlab{linde}
\end{figure}


\subsection{Approximating the fibers}
Next, to approximate the fibers we can proceed as for the Michaelis-Menten-Henri model. As above, one might be alerted by the fact that the part of $X_1(x,y_1)$ which is linear in $y_1$, $\mu_0y_1=\partial_y X_1(x_0,0)y_1=x_0y_1$, is not small with respect to $\epsilon$. However, this causes no problems whatsoever, we can just proceed by replacing $\epsilon \phi$ by $\phi$ and consider \eqref{phieqn} in the form
\begin{eqnarray*}
\epsilon \partial_x \Lambda_1 \phi_0+\mu_0 -\phi_0 A_1 = 0.
\end{eqnarray*}
We obtain
\begin{eqnarray}
\phi_0(x_0)&=& -{\frac {x_0 \left( 4\,x_0^{6}+20\,x_0^{5}\epsilon+37
\,x_0^{4}{\epsilon}^{2}+34\,x_0^{3}{\epsilon}^{3}+19\,x_0^{2}{\epsilon}^{4}+6\,x_0{\epsilon}^{5}+{\epsilon}^{6}
 \right) }{{\epsilon}^{7}+5\,x_0{\epsilon}^{6}+18\,x_0^{2}{
\epsilon}^{5}+52\,x_0^{3}{\epsilon}^{4}+86\,x_0^{4}{
\epsilon}^{3}+83\,x_0^{5}{\epsilon}^{2}+42\,x_0^{6}
\epsilon+8\,x_0^{7}}}\nonumber\\
&=&-\frac12+\,{\frac {\epsilon}{8x_0}}-\frac{3\epsilon^2}{32x_0^2}+\mathcal O(\epsilon^3),\eqlab{phi0lm}
\end{eqnarray}
with an error
\begin{eqnarray*}
 \mu_1 &=&-\partial_x\phi_0 \epsilon \Lambda_1 = -\frac{\epsilon^2}{x_0}+\mathcal O(\epsilon^3).
\end{eqnarray*}
The function $\phi_0$ is therefore correct up to terms including $\epsilon$, more accurate than expected and indicated by the subscripts. This is again due to the fact that the dominant term $-1/2$ in $\phi_0$ is independent of $x_0$. An additional application gives 
\begin{eqnarray*}
 \phi_1 = \frac{\epsilon^2}{32x_0^2}+\mathcal O(\epsilon^3),
\end{eqnarray*}
and $\phi^1 = \phi_0+\phi_1=-\frac12+\,{\frac {\epsilon}{8x_0}}-\frac{\epsilon^2}{16x_0^2}+\mathcal O(\epsilon^3)$ is correct up to terms including $\epsilon^2$ since the new error
\begin{eqnarray*}
 \mu_2 = -\frac{\epsilon^3}{32x_0^2}+\mathcal O(\epsilon^4),
\end{eqnarray*}
is $\mathcal O(\epsilon^3)$. Through $x_0=x_1+\phi^1(x_1)y_1$, $y=\eta^1(x_0)+y_1$ we have cf. \eqref{tsp}, in terms of the original variables $(x,y)$ in \eqref{lm}, then obtained the following approximation, correct up to terms including $\epsilon^2$, to the tangent space of the fibers
\begin{eqnarray}
v(\epsilon)=\begin{pmatrix}
 \phi^1\\
 1+\partial_x \eta^1 \phi^1
\end{pmatrix}
 =
 \begin{pmatrix}
 -\frac12+\,{\frac {\epsilon}{8x_0}}-\frac{\epsilon^2}{16x_0^2}\\
\frac12+\,{\frac {\epsilon}{8x_0}}+\frac{\epsilon^2}{16x_0^2}
 \end{pmatrix}+\mathcal O(\epsilon^3),\eqlab{vso}
\end{eqnarray}
based at $(x_0,\eta^1(x_0))$. Note how the leading order terms suggest to replace $(x,y)$ by $(w,z)$ through
\begin{eqnarray*}
 \begin{pmatrix}
  x\\
  y
 \end{pmatrix} = \begin{pmatrix}
  w\\
  0
 \end{pmatrix}+v(0)z =\begin{pmatrix}
  1 &-\frac12 \\
  0 &\frac12
 \end{pmatrix}\begin{pmatrix}
 w\\
 z
 \end{pmatrix}.
\end{eqnarray*}
Indeed this is just
\begin{eqnarray}
 \begin{pmatrix}
 w\\
 z
 \end{pmatrix} = \begin{pmatrix}
 x+y\\
 2y
\end{pmatrix},\eqlab{wz2}
\end{eqnarray}
the transformation \eqref{wz} from above that transforms the Lindemann system into \eqref{wzeqn}; a system in standard slow-fast form. Finally, \figref{fiber_exeps0_4_linde}, a figure similar to \figref{fiber_exeps0_4} for the Michaelis-Menten-Henri model, verifies the contraction along the approximated fiber directions.

\begin{figure}[h!]
\begin{center}
{\includegraphics[width=.5\textwidth]{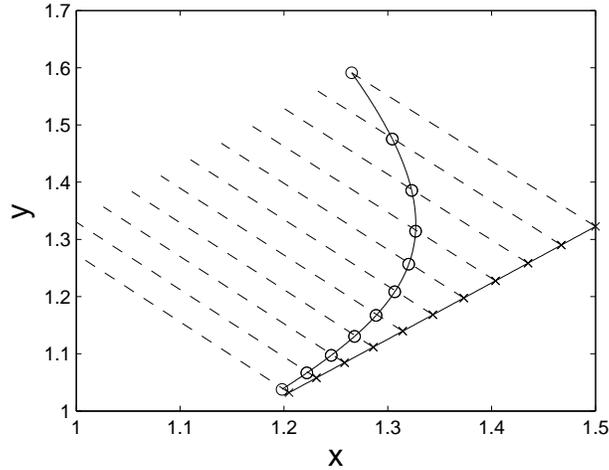}}
\end{center}
\caption{Two solutions $(x_b,y_b)$ ($\times$) and $(x_l,y_l)$ ($\circ$) of \eqref{lm} for $\epsilon=0.4$. The full line represents the slow manifold. The solution $(x_l,y_l)$ is initiated on the fiber corresponding to the base point with $x_{b0}=1.5$, at a distance $s\approx 0.36$ from the base, and is observed to contract to the solution $(x_b,y_b)$ along the fiber directions. The fiber directions obtained through \eqref{vso} are indicated by the dashed lines running from upper left to lower right.}
\figlab{fiber_exeps0_4_linde}
\end{figure}

\subsubsection{Comparison with the CSP method}
For the CSP-approximation of the fiber directions for \eqref{lm} we use Eqs. (3.24)-(3.27) in \cite{kap4}, starting from $A^{(0)} = B_{(0)}=\begin{pmatrix}
                                                                                     0 &1 \\
                                                                                     1 & 0
                                                                                    \end{pmatrix}$ as in \cite{gou1}. 
{We obtain the following matrices:
                                                                                    \begin{align*}
                                                                                     A^{(1)}(x,y) &= \begin{pmatrix}
                                                                                                          A^{(1)}_1(x,y) & A^{(1)}_2(x,y)
                                                                                                         \end{pmatrix}=
\begin{pmatrix}
            -\frac{x}{x+\epsilon} & 1\\
            1-\frac{(2x-y)x}{(x+\epsilon)^2} & \frac{2x-y}{x+\epsilon}
                                                                                               \end{pmatrix},\\
 B_{(1)}(x,y) &=\begin{pmatrix}
                     B_{(1)}^1(x,y)\\
                     B_{(1)}^2(x,y)
                    \end{pmatrix}=
 \begin{pmatrix} 
            -\frac{2x-y}{x+\epsilon} & 1\\
            1-\frac{(2x-y)x}{(x+\epsilon)^2} & \frac{x}{x+\epsilon}
           \end{pmatrix},
                                                                                    \end{align*}
                                                                                    adopting the notation used in \cite{kap4}. 
                                                                                    The first CSP-condition (Eq. (3.29) in \cite{kap4} with $q=1$) then reads
                                                                                    \begin{align*}
                                                                                     B_{(1)}^1(x,\eta_0(x)) \begin{pmatrix}
                                                                                                X(x,\eta^1(x))\\
                                                                                                Y(x,\eta^1(x))\\
                                                                                               \end{pmatrix} = 0,
                                                                                    \end{align*}
                                                                                    giving the first improved slow manifold approximation $y=\eta^1(x)$:
                                                                                    \begin{align*}
                                                                                     \eta^1(x) = {\frac {x^{2} \left( {\epsilon}^{2}+4x\epsilon+2x^{2} \right) 
}{{\epsilon}^{3}+3{\epsilon}^{2}x+5x^{2}\epsilon+2x^{3}}} = x-\frac{1}{2}\epsilon+\mathcal O(\epsilon^2).
                                                                                    \end{align*}
                                                                                    Note again that this $\eta^1$ is the same as the one in \eqref{eta11lm} obtained from the SO-method. This is not true at the following step cf. \figref{linde}. 
                                                                                    To approximate the fiber directions, and finish the first step of the CSP method, we then plug $y=\eta^1(x)$ into the first column of $A^{(1)}$:
                                                                                    \begin{align*}
                                                                                     v = \begin{pmatrix}
            -\frac{x}{x+\epsilon} \\
            1-\frac{(2x-\eta^1(x))x}{(x+\epsilon)^2} 
                                                                                               \end{pmatrix} =\begin{pmatrix}
            -1+\frac{\epsilon}{x}+\mathcal O(\epsilon^2)\\
            \frac{3\epsilon}{2x}+\mathcal O(\epsilon^2)
                                                                                               \end{pmatrix}.
                                                                                    \end{align*}
According to Eq. (3.32) the span of this vector should approximate the tangent spaces to the fibers. It is clearly not in agreement with the SOF-approximation \eqref{vso}. The error is $\mathcal O(1)$. }
An additional application gives more lengthy expressions so we just present the final vector:
%
%
 \begin{eqnarray*}
   v = 
\begin{pmatrix} -\frac12-{\frac {3\epsilon}{8x_0}}\\ \noalign{\medskip}\frac12+{\frac {
\epsilon}{8x}}
\end {pmatrix}+\mathcal O(\epsilon^2),
 \end{eqnarray*}
 whose span cf. Eq. (3.32) approximates the tangent spaces. The error has now been pushed to order $\mathcal O(\epsilon)$ in agreement with the theory of the CSP method. The next approximation gives an error of order $\mathcal O(\epsilon^2)$ and so on. See also \figref{order} where we using Maple have compared the order $n$ of the error for the CSP vectors ($\circ$) and the SOF vectors ($\times$) for the first four applications. Here \textit{order $n$} is understood as the error being $\mathcal O(\epsilon^n)$. This example therefore demonstrates a slower convergence of the fiber approximations for the CSP method (lagging two orders behind) when compared to the approximations obtained from our SOF method. 
\begin{figure}[h!]
\begin{center}
{\includegraphics[width=.5\textwidth,trim = 0 130 0 130]{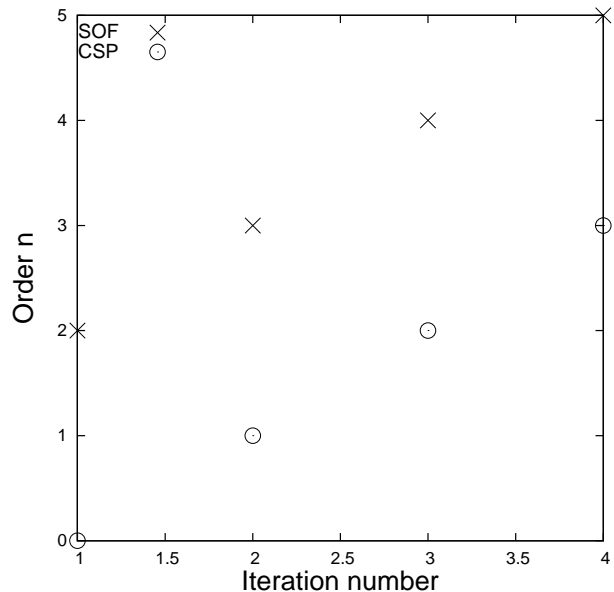}}
\end{center}
\caption{The order $n$ of the error, meaning that the remainder is $\mathcal O(\epsilon^n)$, for the approximation of the fibers using the CSP method ($\circ$) and the SOF method ($\times$) for the first four iterations. In this example, the order for the CSP approximation is two less than the SOF approximation for a given iteration.  }
\figlab{order}
\end{figure}

\section{Conclusion}
In this paper we developed a new method, the SOF method, as an extension of the method of straightening out (SO method) so that it can also be used to approximate fiber directions. The method is based on normal form computations. After having approximated the slow manifold using the SO method, the extended method constructs a transformation of the slow variables as a product of a finite 
 sequence of transformations, each obtained as the solution of a linear
equation, so that the slow dynamics becomes almost independent of the fast
variable to linear order. See \thmref{thm1}. The extended method preserves the unique properties of the SO method such as: (i) It approximates exponentially well. (ii) The method can be written only in terms of the vector-field and its Jacobian matrix, making it suitable for numerical implementation. (iii) No smoothness in $\epsilon$ is required. (iv) The approximations improves near equilibria. Neither naive asymptotic expansions, the ILDM method, the CSP method nor the ZDP method possess all of the properties. In particular, the $\mathcal O(\epsilon^n)$-estimates from the analysis of the CSP and ZDP methods have not yet been improved to exponential ones. Perhaps more importantly, if you were to write out the equations describing the $n$th step of the CSP and ZDP methods they would involve derivatives up to order $n$.  
Our method was successfully applied to two classical examples: the Michaelis-Menten-Henri model and the Lindemann mechanism. The latter demonstrated the use of the method for a system in non-standard slow-fast form. In this example the SOF method gave more accurate approximations of the fiber directions when compared with the approximations obtained using the CSP method. The approach
was also extended further in \thmref{thm2} so that it can be used to approximate, again
exponentially well, the curvature of the fibers. This result holds true
even when the slow manifold is of saddle type. In the saddle case, we only
require that it is not neutral in the sense that there does not exist a
contraction and an expansion rate of equal magnitude. 
\newpage
\bibliography{refs}

\begin{thebibliography}{10}

\bibitem{bat1}
P.W. Bates, K.~Lu, and C.~Zeng.
\newblock {Approximately invariant manifolds and global dynamics of spile
  states}.
\newblock {\em Inventiones mathematicae}, 174:355--433, 2008.

\bibitem{templatorref}
K.~M. Beutel and E.~Peacock-L\'opez.
\newblock {Complex dynamics in a cross-catalytic self-replication mechanisms}.
\newblock {\em Journal of Chemical Physics}, 126(125104), 2007.

\bibitem{bro2}
M.~Br{\o}ns.
\newblock {Canard explosion of limit cycles in templator models of
  self-replication mechanisms}.
\newblock {\em Journal of Chemical Physics}, 134(144105), 2011.

\bibitem{chi1}
C.~Chicone.
\newblock {\em Ordinary Differential Equations with Applications}, volume~34.
\newblock Texts in Applied Mathematics, Springer, 2006.

\bibitem{cotrei1}
C.~J. Cotter and S.~Reich.
\newblock {Adiabatic invariance and applications to MD and NWP}.
\newblock {\em BIT Numerical Mathematics}, 3:439--455, 2003.

\bibitem{ajr2}
S.~M. Cox and A.~J. Roberts.
\newblock {Initial conditions for models of dynamical systems}.
\newblock {\em Physica D}, 85:126--141, 1995.

\bibitem{fen1}
N.~Fenichel.
\newblock Persistence and smoothness of invariant manifolds for flows.
\newblock {\em Indiana University Mathematics Journal}, 21:193--226, 1971.

\bibitem{fen2}
N.~Fenichel.
\newblock Asymptotic stability with rate conditions.
\newblock {\em Indiana University Mathematics Journal}, 23:1109--1137, 1974.

\bibitem{fra1}
S.J. Fraser.
\newblock {The steady state and equilibrium approximations: a geometrical
  picture}.
\newblock {\em Journal of Chemical Physics}, 88:4732--4738, 1990.

\bibitem{kap1}
C.~W. Gear, T.~J. Kaper, I.~G. Kevrekidis, and A.~Zagaris.
\newblock {Projecting to a slow manifold: Singularly Perturbed Systems and
  Legacy Codes}.
\newblock {\em SIAM Journal of Applied Dynamical Systems}, 4(3):711--732, 2005.

\bibitem{geller2}
V.~Gelfreich and L.~Lerman.
\newblock {Long-periodic orbits and invariant tori in a singularly perturbed
  Hamiltonian system}.
\newblock {\em Physica D}, 176:pp 125--146, 2003.

\bibitem{geller1}
V.~Gelfriech and L.~Lerman.
\newblock {Almost invariant elliptic manifold in a singularly perturbed
  Hamiltonian system}.
\newblock {\em Nonlinearity}, 15:447--557, 2002.

\bibitem{gou1}
D.~A. Goussis and M.~Valorani.
\newblock {An efficient iterative algorithm for the approximation of the fast
  and slow dynamics of stiff systems}.
\newblock {\em Journal of Computational Physics}, 214:316--346, 2006.

\bibitem{guck5}
J.~Guckenheimer, K.~Hoffman, and W.~Weckesser.
\newblock {The forced van der Pol equation I: The slow flow and its
  bifurcations}.
\newblock {\em SIAM Journal of Applied Dynamical Systems}, 2:1--35, 2003.

\bibitem{guc1}
J.~Guckenheimer and P.~Holmes.
\newblock {\em Nonlinear Oscilations, Dynamical Systems, and Bifurcations of
  Vector Field}, volume~42.
\newblock New York: Springer-Verlag, 1983.

\bibitem{guc3}
J.~Guckenheimer and C.~Kuehn.
\newblock {Computing Slow Manifolds of Saddle Type}.
\newblock {\em SIAM Journal of Applied Dynamical Systems}, 8(3):854--879, 2009.

\bibitem{guck6}
J.~Guckenheimer and C.~Kuehn.
\newblock {Homoclinic orbits of the FitzHugh-Nagumo equation: Bifurcations in
  the full system}.
\newblock {\em SIAM Journal of Applied Dynamical Systems}, 9:138--153, 2010.

\bibitem{har1}
M.~Haragus and G.~Iooss.
\newblock {\em {Local Bifurcations, Center Manifolds and Normal Forms in
  Infinite Dimensional Dynamical Systems}}.
\newblock Universitext, Springer London, 2011.

\bibitem{hor1}
R.~A. Horn and C.~R. Johnson.
\newblock {\em Topics in Matrix Analysis}.
\newblock Cambridge University Press, 1991.

\bibitem{ioo1}
G.~Iooss and E.~Lombardi.
\newblock {Approximative invariant manifolds up to exponentially small terms}.
\newblock {\em Journal of Differential Equations}, 6:1410--1431, 2010.

\bibitem{jon1}
C.K.R.T. Jones.
\newblock {\em {Geometric Singular Perturbation Theory, Lecture Notes in
  Mathematics, Dynamical Systems (Montecatini Terme)}}.
\newblock Springer, Berlin, 1995.

\bibitem{kap3}
H.~G. Kaper and T.~J. Kaper.
\newblock Asymptotic analysis of two reduction methods for systems of chemical
  reactions.
\newblock {\em Physica D}, 165:66--93, 2002.

\bibitem{lam1}
S.~H. Lam.
\newblock {Using CSP to understand complex chemical kinetics}.
\newblock {\em Combustion, Science and Technology}, 89:375--404, 1993.

\bibitem{lam2}
S.~H. Lam and D.~A. Goussis.
\newblock {Understanding complex chemical kinetics with computational singular
  perturbation}.
\newblock {\em Proceedings of the 22nd International Symposium on Combustion,
  Seattle, WA}, pages 931--941, 1988.

\bibitem{las1}
J.~Laskar.
\newblock {Large scale chaos in the Solar System}.
\newblock {\em Astronomy and Astrophysics}, 287:9--12, 1994.

\bibitem{las2}
J.~Laskar and M.~Gastineau.
\newblock {Existence of collisional trajectories of Mercury, Mars and Venus
  with the Earth}.
\newblock {\em Nature}, 459:817--819, 11 June 2009.

\bibitem{lor2}
E.~N. Lorenz.
\newblock {The slow manifold - what is it?}
\newblock {\em American Meteorological Society}, 15 December, 1992.

\bibitem{lor1}
E.~N. Lorenz.
\newblock {Existence of a slow manifold}.
\newblock {\em Journal of the Atmospheric Sciences}, 43(15):1547--1557, 1986.

\bibitem{lor3}
E.~N. Lorenz and V.~Krishnamurty.
\newblock {On the non-existence of a slow manifold}.
\newblock {\em Journal of the Atmospheric Sciences}, 44:2940--2950, 1987.

\bibitem{maa1}
U.~Maas and S.~B. Pope.
\newblock {Simplifying chemical kinetics: Intrinsic low-dimensional manifolds
  in composition space}.
\newblock {\em Combustion and Flame}, 88:239--264, 1992.

\bibitem{mac1}
R.~S. MacKay.
\newblock Slow manifolds.
\newblock {\em In: ``Energy Localisation and Transfer'', eds T Dauxois, A
  Litvak-Hinenzon, RS MacKay, A Spanoudaki, World Scientific}, pages 149--192,
  2004.

\bibitem{man1}
R.~Ma{\~n}e.
\newblock {Persistent manifolds are normally hyperbolic}.
\newblock {\em Bulletin of American Mathematical Society}, 80:90--91, 1980.

\bibitem{McQ1}
D.~A. McQuarrie.
\newblock {\em {Physical Chemistry: A Molecular Approach}}.
\newblock Sausalito: University Science Books, 1997.

\bibitem{micmen1}
L.~Michaelis and M.~Menten.
\newblock {Die Kinetik der Invertinwirkung}.
\newblock {\em Biochemische Zeitschrift}, 49:333--369, 1913.

\bibitem{murd1}
J.~Murdock.
\newblock {\em {Normal forms and Unfoldings for Local Dynamical Systems}}.
\newblock Springer Monographs in Mathematics, Springer New York, 2003.

\bibitem{Mur1}
J.~D. Murray.
\newblock {\em Mathematical Biology}, volume~19.
\newblock Berlin: Springer-Verlag, 1993.

\bibitem{naf1}
J.~Nafe and U.~Mass.
\newblock {A general algorithm for improving ILDMs}.
\newblock {\em Combustion Theory and Modelling}, 6:697--709, 2002.

\bibitem{nei2}
A.~Neishtadt.
\newblock On the accuracy of conservation of adiabatic invariant.
\newblock {\em Journal of Applied Mathematics and Mechanics}, 45(1):58--63,
  1982.

\bibitem{nei3}
A.~Neishtadt.
\newblock The separation of motions in systems with rapidly rotating phase.
\newblock {\em Journal of Applied Mathematics and Mechanics}, 48(2):133--139,
  1984.

\bibitem{nei87}
A.~Neishtadt.
\newblock Persistence of stability loss for dynamical bifurcation, i.
\newblock {\em Differential Equations}, 23:1385--1390, 1987.

\bibitem{pos2}
J.~Poschel.
\newblock { A lecture on the classical KAM-theorem}.
\newblock {\em Proceedings of Symposia in Pure Mathematics}, 69:707--732, 2001.

\bibitem{pos1}
J.~Poschel and E.~Trubowitz.
\newblock {\em {Inverse Spectral Theory}}, volume 130.
\newblock Pure and Applied Mathematics, Academic Press, Inc., 1987.

\bibitem{rob89}
A.~J. Roberts.
\newblock {Appropriate initial conditions for asymptotic describtion of the
  long term evolution of dynamical systems}.
\newblock {\em Journal of Australian Mathematics, Series B}, 31:48--75, 1989.

\bibitem{ajr1}
A.~J. Roberts.
\newblock {Computer algebra derives correct initial conditions for
  low-dimensional dynamical models}.
\newblock {\em Computer Physics Communications}, 126:187--206, 2000.

\bibitem{romtur1}
V.~Rom-Kedar and D.~Turaev.
\newblock The symmetric parabolic resonance instability.
\newblock {\em Nonlinearity}, 23:1325--1351, 2010.

\bibitem{fra3}
M.R. Roussel.
\newblock {Forced convergence iterative schemes for the approximation of
  invariant manifolds}.
\newblock {\em Journal of Mathematical Chemistry}, 21:385--393, 1997.

\bibitem{fra2}
M.R. Roussel and S.J. Fraser.
\newblock {Geometry of the steady-state approximation: perturbation and
  accelerated convergence method}.
\newblock {\em Journal of Chemical Physics}, 93:1072--1081, 1990.

\bibitem{dav1}
R.T. Skodje and M.J. Davis.
\newblock {Geometric simplification of complex kinetic systems}.
\newblock {\em Journal of Physical Chemistry A}, 105:10356--10365, 2001.

\bibitem{tak1}
F.~Takens and A.~Vanderbauwhede.
\newblock {\em {Handbook of Dynamical Systems. 3. Local invariant manifolds and
  Normal forms}}, volume~3.
\newblock Elsevier B.V., 2010.

\bibitem{tem1}
R.~Temam.
\newblock {Inertial Manifolds}.
\newblock {\em The Mathematical Intelligencer}, 12(4):68--74, 1990.

\bibitem{kri4}
K.~Uldall~Kristiansen, M.~Br{\o}ns, and J.~Starke.
\newblock Computation of slow manifolds and transients using iterative methods.
\newblock {\em Manuscript}, 2013.

\bibitem{kri1}
K.~Uldall~Kristiansen, P.~Palmer, and R.~M. Roberts.
\newblock A unification of models of tethered satellites.
\newblock {\em SIAM Journal of Applied Dynamical Systems}, 10:1042--1069, 2011.

\bibitem{kri2}
K.~Uldall~Kristiansen, P.~Palmer, and R.~M. Roberts.
\newblock The persistence of a slow manifold with bifurcation.
\newblock {\em SIAM Journal of Applied Dynamical Systems}, 11:661--683, 2012.

\bibitem{kri3}
K.~Uldall~Kristiansen and C.~Wulff.
\newblock Exponential estimates of slow manifolds.
\newblock {\em arXiv:1208.4219v1 [math.DS]}, 2012.

\bibitem{yoc1}
J.-C. Yoccoz.
\newblock { Siegel theorem, Bryuno numbers and quadratic polynomials}.
\newblock {\em Asterisque}, 231:265--293, 1995.

\bibitem{kap2}
A.~Zagaris, C.~W. Gear, T.~J. Kaper, and I.~G. Kevrekidis.
\newblock {Analysis of the accuracy and convergence of equation-free projection
  to a slow manifold}.
\newblock {\em ESAIM: Mathematical Modelling and Numerical Analysis},
  43:757--784, 2009.

\bibitem{kap4}
A.~Zagaris, H.~G. Kaper, and T.~J. Kaper.
\newblock {Fast and slow dynamics for the CSP method}.
\newblock {\em SIAM Journal of Multiscale Modelling and Simulation},
  2:613--638, 2004.

\end{thebibliography}
\bibliographystyle{plain}
\newpage 
 \end{document}